\documentclass[preprint,12pt]{elsarticle}

\usepackage{graphicx}
\usepackage{amssymb}
 \usepackage{amsthm}
 \usepackage{amsmath}
 \usepackage{braket,amsfonts}
 \usepackage{array}
\usepackage{enumitem}
\usepackage{algorithmic}
\usepackage{graphicx,epstopdf}
\usepackage{amsbsy}
\usepackage[english]{babel}
\usepackage{blindtext}
\newtheorem{theorem}{Theorem}
\newtheorem{prop}{Proposition}

\usepackage{lineno}

\journal{arxiv}

\begin{document}

\begin{frontmatter}


\title{A Numerical Method for the Parametrization of Stable and Unstable Manifolds of Microscopic Simulators}



\author{Constantinos Siettos\corref{cor1}}
\address{Universit\'a degli Studi di Napoli Federico II, Dipartimento di Matematica e Applicazioni ``Renato Caccioppoli", Naples, Italy}

\author{Lucia Russo}
\address{Consiglio Nazionale delle Ricerche, Naples, Italy}

\cortext[cor1]{Corresponding author: Constantinos Siettos, email:constantinos.siettos@unina.it}

\begin{abstract}
We address a numerical methodology for the computation of coarse-grained stable and unstable manifolds of saddle equilibria/stationary states of multiscale/stochastic systems for which a ``good" macroscopic description in the form of Ordinary (ODEs) and/or Partial differential equations (PDEs) does not explicitly/ analytically exists in a closed form. Thus, the assumption is that we have a detailed microscopic simulator of a complex system in the form of Monte-Carlo, Brownian dynamics, Agent-based models e.t.c. (or a black-box large-scale discrete time simulator) but due to the inherent complexity of the problem, we don't have explicitly an accurate model in the form of ODEs or PDEs.
Our numerical scheme is a three-tier one including: (a) the ``on demand" detection of the coarse-grained saddle equilibrium, (b) its coarse-grained stability analysis, and (c) the parametrization of the semi-local invariant stable and unstable manifolds by the numerical solution of the homological/functional equations for the coefficients of the truncated series approximation of the manifolds.
\end{abstract}


\end{frontmatter}


\section{Introduction}\label{sec1}

The computation of invariant manifolds of dynamical systems is very important for a series of system-level tasks, particularly for the bifurcation analysis and control. For example, the detection of stable manifolds of saddle points allows the identification of the boundary between different basins of attraction, while the intersection of stable and unstable manifolds most-often leads to complex dynamical behaviour such as chaotic dynamics \cite{Yagasaki1994,Feudel2005}. Their computation is also central to the control of nonlinear systems and especially in the control of chaos \cite{Ott1990,Schiff1994,Yagasaki1994,Flakamp2017}.
However, their computation is not trivial: even for relatively simple low-dimensional ODEs, their analytical derivation is most of the times an overwhelming difficult task. Thus, one has to resort to their numerical approximation. However, this task is not easy; at the beginning of '90s only one-dimensional global invariant manifolds of vector fields could be computed. Guckenheimer \& Worfolk \cite{Guckenheimer1993} proposed an algorithm for converging on the stable manifold of saddles based on geodesics emanating from the saddle by iteratively rescaling the radial part of the vector field on the submanifold. Johnson et al. (1997) \cite{Johnson1997} introduced a numerical scheme to reconstruct two-dimensional stable and unstable manifolds of saddles. The proposed method starts with the creation of a ring of points on the local-linear eigenspace and successively creates circles of points that are then connected by a triangular mesh. The appropriate points are selected through time integration so that the velocity of the vector field is similar in aN arc-length sense for all trajectories. Krauskopf \& Osinga (1999) \cite{Krauskopf1999} developed a numerical method based on geodesics; the manifold is evolved iteratively by hyperplanes perpendicular to a previous detected geodesic circle. Krauskopf et al. (2005) \cite{KRAUSKOPF2005} addressed a numerical method for the approximation of two-dimensional stable and unstable manifolds which incorporates the solution of a boundary value problem; the method performs a continuation of a family of trajectories possessing the same  arc-length. For a survey of methods for the numerical computation of stable and unstable manifolds see also Krauskopf et al. (2005) \cite{KRAUSKOPF2005}. In the above methods, the stable manifold is computed as the unstable manifold of the inverse map, i.e. by following the flow of the vector field backward in time \cite{England2004}. Thus, an explicit knowledge of the vector field and its inverse is required which  however is not always available. England et al. (2004) \cite{England2004} presented an algorithm for computing one-dimensional stable manifolds for planar maps when an explicit expression for the inverse map is not available and/or even the map is not invertible. Triandaf et al. (2003) \cite{Triandaf2003}  proposed a procedure for approximating stable and unstable manifolds given only experimental data based on time-delay embeddings of a properly selected data set of initial conditions.\\
Another approach to compute invariant manifolds, the so-called parametrization method has been introduced by Cabre et al. \cite{Cab1,Cab2,Cabr3}. This is a numerical-assisted approach based on functional analysis tools for deriving analytical expressions of the semi-local invariant manifolds. This involves the expansion of the invariant manifold as series and the construction of a system of homological equations for the coefficients of the series. Based on this approach, Haro et al. (2016) \cite{Haro2016} addressed a numerical approach for the computation of the coefficients of high order power series expansions of parametrizations of two-dimensional invariant manifolds. Breden et al. (2016) \cite{Breden2016} employed the parametrization method to compute stable and unstable manifolds of vectors fields. For the implementation of the method it is assumed that the vector field is explicitly available in a closed form.\\
However, for many complex systems of contemporary interest, the equations that can describe adequately the dynamics at the macroscopic-continuum scale are not explicitly available in the form of ODEs or PDEs in a closed form. Take for example the case where the laws that govern the dynamics of the interactions between the units that constitute the system may be known in the form of e.g. molecular dynamics, Brownian dynamics, agent-based modeling, Monte Carlo etc., but a ``good" macroscopic description is not available in a closed form. For this kind of problems the lack of a macroscopic description in a closed form constitutes a stumbling block in our ability to systematically analyse, design and control the emergent dynamics. Two ways are traditionally used to study the emergent behaviour of such microscopic dynamical models. On the one hand, there is the simple temporal simulation. An ensemble of many initial conditions would be set up; a large enough number of ensemble realizations would be created for each initial condition; some of the parameters of the model would probably have to be modified and finally the statistics of the detailed dynamics of the system would be monitored for a long time to investigate the coarse-grained behaviour. However, this ``simple" temporal simulation is most of the times inappropriate  for the systematic bifurcation analysis, optimization and control of the emergent behaviour. On the other hand there is the statistical-mechanics/assisted approach where one tries to analytically find closures, i.e. the relations for the moments of the detailed microscopic distribution that would allow the derivation of evolution equations at the macrosocpic/emergent level. For example, for Monte Carlo simulations (these processes are typically Markovian) a Master Equation can be derived from which evolution equations are obtained for a few moments of the underlying probability distribution.  However, these equations usually involve higher-order moments whose evolution dynamics are functions of  higher order moments. This lead to an infinite hierarchy of evolution equations. Thus at some level these higher order moments have to be expressed as functions of the lower-order ones in order to close the system of equations. However, the assumptions that underlie these ``closures" introduce certain qualtitative and quantitative biases in the analysis of the ``actual" system as represented by the best available microsocpic simulator (see for example in \cite{Reppas2012} a comparative analysis between various closures for a microscopic model and a discussion about the biases that are introduced).\\
The Equation-free approach \cite{Makeev2002,Kevrekidis2003,Siettos2003,Kevrekidis2004}, a multiscale numerical-assisted framework, allows the establishment of the link between traditional continuum numerical analysis and microscopic/ stochastic simulation of complex/multiscale systems. The Equation-Free approach allows the systematic numerical analysis of the coarse-grained macrosocpic dynamics bypassing the derivation of ``closures" in an explicit analytically form. The method identifies ``on-demand" the quantities required for performing numerical analysis at the continuum level, such as coarse-grained Jacobians and Hessians; these quantities are obtained by appropriately initialized runs of the microsocpic simulators, which are treated as black boxes maps. Regarding the computation of coare-grained invariant manifolds, Gear and Kevrekidis \cite{GearKevrekidis} introduced a method for the convergence on the coarse-grained slow manifolds of legacy simulators by requiring that the change in the ``fast" variables (i.e. the variables that are quickly ``slaved" to the variables that parametrize the slow manifold) is zero. In another paper, Gear et al. \cite{GearKaper} computed coarse-grained
slow manifolds by restricting the derivatives of the ``fast" variables to zero. Zagaris et al. (2009) \cite{Zagaris2009} performed a systematic analysis of the accuracy and convergence of Equation-free projection to the slow manifold.\\
Here, we present a new numerical method for the computation of coarse-grained stable and unstable manifolds of saddle equilibria/stationary states of microscopic dynamical simulators (and in general discrete-time black-box maps). Our method is based on the Equation-free framework. The approximation of the semi-local coarse-grained stable and unstable manifolds is achieved  by a truncated polynomial expansion; the coefficients of the series are computed by the Newton-Raphson method applied on a coarse-grained map of the microscopic simulator. Thus, the proposed numerical method involves a three-step procedure including the Equation-free: (a) detection of the coarse-grained saddle (b) computation of the coarse-grained Jacobian on the saddle and the computation of the corresponding eigenmodes,  (c) identification of the polynomial coefficients of the semi-local coarse-grained stable and unstable manifolds; this step involves (i) the numerical construction of a back-box coarse-grained map for the coefficients of the polynomial series, (ii) iterative estimation of the polynomial coefficients by applying Newton's method around the constructed coarse-grained map.
The method is illustrated through two examples whose stable and unstable manifolds are also approximated analytically through the parametrization method for accessing the efficiency of our proposed numerical method. The first example is a simple toy discrete-time map and the second one is a Gillespie-Monte Carlo realization of a simple catalytic reaction scheme describing the dynamics of CO oxidation on catalytic surfaces.

\section{Computation of Stable \& Unstable Manifolds of Saddles for Discrete-time Models}
\label{sec:Discrete}

We will first present the way for approximating the stable and unstable manifolds of a saddle point for discrete-time systems when the equations are given in an explicit form. Then, we will show how one can approximate the stable and unstable manifolds when equations are not given in an explicit form. The later case includes large-scale black-box simulators as well as microscopic/stochastic multiscale models.

Let us consider the discrete-time model given by:

\begin{equation}
\boldsymbol{x}_{k+1} =\boldsymbol{F}(\boldsymbol{x}_k, \boldsymbol{p})  
\label{eqn1}
\end{equation}

where $\boldsymbol{F}:R^n \times R^m \rightarrow R^n$ is a smooth multivariable, vector-valued function having $\boldsymbol{x}_k \in R^n$ as initial condition.

Regarding the computation of the stable and unstable manifolds of the saddle point of the above discrete-time system, We will prove the following theorem:

\begin{theorem} Let us denote by $(\boldsymbol{x}^*, \boldsymbol{p}^*$) the saddle fixed point of the discrete-time model (\ref{eqn1}) which satisfies $\boldsymbol{x}^*=\boldsymbol{F}(\boldsymbol{x}^*, \boldsymbol{p}^*)$. Let us also assume that the Jacobian $\nabla \boldsymbol{F}(\boldsymbol{x}^*, \boldsymbol{p}^*)$ is diagonalizable. Let $\boldsymbol{V}_1$ be the $n \times l$ matrix whose columns are the eigevectors of $\nabla \boldsymbol{F}(\boldsymbol{x}^*, \boldsymbol{p}^*)$ that correspond to the $l$ eigenvalues lying inside the unit circle, and $\boldsymbol{V}_2$ be the $n \times n-l$ matrix whose columns are the eigevectors of $\nabla \boldsymbol{F}(\boldsymbol{x}^*, \boldsymbol{p}^*)$ that correspond to the $n-l$ eigenvalues lying outside the unit circle. Let us also define $\boldsymbol{z}_s \in R^l$ and $\boldsymbol{z}_u \in R^{n-l}$ by the transformation $\boldsymbol{x}'=\begin{bmatrix} \boldsymbol{V}_1 & \boldsymbol{V}_2  \end{bmatrix} \begin{bmatrix} \boldsymbol{z}_s \\ \boldsymbol{z}_u  \end{bmatrix}$, where $\boldsymbol{x}'=\boldsymbol{x}-\boldsymbol{x}^*$. Then the fixed point $\boldsymbol{x}'=\boldsymbol{0}$ has: 

\item (A1) a $C^r$ $l$-dimensional local stable manifold $\boldsymbol{W}_s(\boldsymbol{0})$ tangent to the subspace spanned by the columns of $\boldsymbol{V}_1$ at the origin defined by:

\begin{equation}
\boldsymbol{W}_s(\boldsymbol{0})=\{(\boldsymbol{z}_s,\boldsymbol{z}_u)\in R^l \times R^{n-l}| \boldsymbol{z}_{u}=\boldsymbol{h}_s(\boldsymbol{z}_{s})\},
\label{eqn2}
\end{equation}

where $\boldsymbol{h}_s:R^{l}\rightarrow R^{n-l}$ is a $C^r$ function which satisfies $\boldsymbol{h}_s(\boldsymbol{z}_{s})=\boldsymbol{0}$ and $\nabla_{\boldsymbol{z}_s}h_j\equiv(\frac{\partial h_{sj}}{\partial z_{s1}},\frac{\partial h_{sj}}{\partial z_{s2}},\dots \frac{\partial h_{sj}}{\partial z_{sl}})=\boldsymbol{0}$, $\forall h_{sj}, j=1,2,\dots n-l$; $h_{sj}(\boldsymbol{z}_s$)  is the $j$-th component of  $\boldsymbol{h_s(\boldsymbol{z}_s)}$.\\

\item (A2) a $C^r$ $n-l$-dimensional local stable manifold $\boldsymbol{W}_u(\boldsymbol{0})$ tangent to the subspace spanned by the columns of $\boldsymbol{V}_2$ at the origin defined by:

\begin{equation}
\boldsymbol{W}_u(\boldsymbol{0})=\{(\boldsymbol{z}_s,\boldsymbol{z}_u)\in R^l \times R^{n-l}| \boldsymbol{z}_{s}=\boldsymbol{h}_u(\boldsymbol{z}_{u})\},
\label{eqn3}
\end{equation}

where $\boldsymbol{h}_u:R^{n-l}\rightarrow R^{l}$ is a $C^r$ function which satisfies $\boldsymbol{h}_u(\boldsymbol{z}_{u})=\boldsymbol{0}$ and $\nabla_{\boldsymbol{z}_u}h_{uj}\equiv(\frac{\partial h_{uj}}{\partial z_{u1}},\frac{\partial h_{uj}}{\partial z_{u2}},\dots \frac{\partial h_{uj}}{\partial z_{u1}})=\boldsymbol{0}$, $\forall h_{uj}, j=1,2,\dots l$; $h_{uj}(\boldsymbol{z}_u$)  is the $j$-th component of  $\boldsymbol{h_u(\boldsymbol{z}_u)}$.\\

\item (B1) On the stable manifold the following system of functional equations hold:

\begin{equation}
\begin{aligned}
\boldsymbol{h}_s(\boldsymbol{\Lambda}_s\boldsymbol{z}_{s}+\boldsymbol{g}_s(\begin{bmatrix} \boldsymbol{V}_1 & \boldsymbol{V}_2  \end{bmatrix}\begin{bmatrix} \boldsymbol{z_s} \boldsymbol{h}_s(\boldsymbol{z}_{s})\end{bmatrix},\boldsymbol{x}^*,\boldsymbol{p}^*))=\\
\boldsymbol{\Lambda}_u\boldsymbol{h}_s(\boldsymbol{z}_{s})+\boldsymbol{g}_u(\begin{bmatrix} \boldsymbol{V}_1 & \boldsymbol{V}_2  \end{bmatrix}\begin{bmatrix} \boldsymbol{z_s}\\\boldsymbol{h}_s(\boldsymbol{z}_{s})\end{bmatrix},\boldsymbol{x}^*,\boldsymbol{p}^*)) 
\label{eqn4}
\end{aligned}
\end{equation}
\\
\item (B2) On the unstable manifold the following system of functional equations hold:

\begin{equation}
\begin{aligned}
\boldsymbol{h}_u(\boldsymbol{\Lambda}_u\boldsymbol{z}_{u}+\boldsymbol{g}_u(\begin{bmatrix} \boldsymbol{V}_1 & \boldsymbol{V}_2  \end{bmatrix}\begin{bmatrix} \boldsymbol{h}_u(\boldsymbol{z}_{u}) \boldsymbol{z}_{u}\end{bmatrix},\boldsymbol{x}^*,\boldsymbol{p}^*))=\\
\boldsymbol{\Lambda}_s\boldsymbol{h}_u(\boldsymbol{z}_{u})+\boldsymbol{g}_s(\begin{bmatrix} \boldsymbol{V}_1 & \boldsymbol{V}_2  \end{bmatrix}\begin{bmatrix} \boldsymbol{h}_u(\boldsymbol{z}_{u})\\\boldsymbol{z}_u\end{bmatrix},\boldsymbol{x}^*,\boldsymbol{p}^*)) 
\label{eqn5}
\end{aligned}
\end{equation}
\\

In the above, $\boldsymbol{\Lambda}_s$ is the $l\times l$ (block) diagonal matrix containing the $l$ eigenvalues with $|\lambda_{i}|<1$ and $\boldsymbol{\Lambda}_u$ is the $(n-l)\times (n-l)$ (block) diagonal matrix containing the $(n-l)$ eigenvalues with $|\lambda_{i}|>1$; $\boldsymbol{g}_s$ and $\boldsymbol{g}_u$ are $l$ and $n-l$ vector-valued functions, respectively, obtained by 

\begin{equation}
\begin{bmatrix} \boldsymbol{g}_s \\ \boldsymbol{g}_u  \end{bmatrix}=\begin{bmatrix} \boldsymbol{V}_1 & \boldsymbol{V}_2  \end{bmatrix}^{-1}\boldsymbol{g}(\boldsymbol{x}',\boldsymbol{x}^*,\boldsymbol{p}^*)
\label{eqn6}
\end{equation}

where $\boldsymbol{g}(\boldsymbol{x}',\boldsymbol{x}^*,\boldsymbol{p}^*)$ corresponds to the $n$-vector-valued nonlinear function: 

\begin{align}
    \boldsymbol{g}(\boldsymbol{x}',\boldsymbol{x}^*,\boldsymbol{p}^*) &= \begin{bmatrix}
          g_{1}(\boldsymbol{x}',\boldsymbol{x}^*,\boldsymbol{p}^*)\\
          g_{2}(\boldsymbol{x}',\boldsymbol{x}^*,\boldsymbol{p}^*) \\
           \vdots \\
          g_{n}(\boldsymbol{x}',\boldsymbol{x}^*,\boldsymbol{p}^*)
         \end{bmatrix}
         \label{eqn7}
  \end{align}

containing all, but the linearization around the saddle, non-linear terms of $\boldsymbol{F}(\boldsymbol{x}_k, \boldsymbol{p})$ satisfying: $ \left\lVert \boldsymbol{g}(\boldsymbol{x}',\boldsymbol{x}^*, \boldsymbol{p}^*)\right\rVert \leq c(\boldsymbol{x}^*) \left\lVert \boldsymbol{x}' \right\rVert^2$.

\end{theorem}
\begin{proof}
\item (A1), (A2)  

By taking $\boldsymbol{x}'=\boldsymbol{x}-\boldsymbol{x}^*$, the model given by Eq.~(\ref{eqn1}) reads:

\begin{equation}
\boldsymbol{x}'_{k+1} =-\boldsymbol{x}^*+\boldsymbol{F}(\boldsymbol{x}'_k+\boldsymbol{x}^*, \boldsymbol{p}^*)  
\label{eqn8}
\end{equation}

By assuming that the vector field:

\begin{align}
    \boldsymbol{F}(\boldsymbol{x},\boldsymbol{p}) &= \begin{bmatrix}
           F_{1}(\boldsymbol{x},\boldsymbol{p})\\
           F_{2}(\boldsymbol{x},\boldsymbol{p}) \\
           \vdots \\
          F_{n}(\boldsymbol{x},\boldsymbol{p})
         \end{bmatrix}
         \label{eqn9}
  \end{align}
  
is differentiable on an open ball $B$ around $\boldsymbol{x}'$ and $\boldsymbol{x}^*$  the right-hand-side of Eq.~(\ref{eqn8}) around  $\boldsymbol{x}^*$ can be written as:

\begin{equation}
\boldsymbol{x}'_{k+1} =\nabla \boldsymbol{F}(\boldsymbol{x}^*, \boldsymbol{p}^*)\boldsymbol{x}' + \boldsymbol{g}(\boldsymbol{x}',\boldsymbol{x}^*, \boldsymbol{p}^*) 
\label{eqn10},
\end{equation}

where $\nabla \boldsymbol{F}(\boldsymbol{x}^*, \boldsymbol{p}^*)$ is the Jacobian evaluated at $(\boldsymbol{x}^*, \boldsymbol{p}^*)$ and $\boldsymbol{g}(\boldsymbol{x}',\boldsymbol{x}^*, \boldsymbol{p}^*)$ contains all the higher order non-linear terms of $\boldsymbol{F}(\boldsymbol{x}_k, \boldsymbol{p})$. For example if $\boldsymbol{F}(\boldsymbol{x}_k, \boldsymbol{p})$ is expanded in a Taylor expansion around the saddle point then:

\begin{equation}
\begin{aligned}
\boldsymbol{g}(\boldsymbol{x}',\boldsymbol{x}^*, \boldsymbol{p}^*) =\frac{1}{2}\sum_{i,j=1}^{n} \boldsymbol{F}_{x_ix_j}(\boldsymbol{x}^*, \boldsymbol{p}^*)(x_i-x_i^*)(x_j-x_j^*)+\dots \\
+\frac{1}{k!}\sum_{i_1,i_2, \dots i_k=1}^{n} \boldsymbol{F}_{x_{i_1}x_{i_2}\dots x_{i_k}}(\boldsymbol{x}^*, \boldsymbol{p}^*)(x_{i_1}-x_{i_1}^*) \cdot (x_{i_k}-x_{i_k}^*)+O(\left\lVert \boldsymbol{x}'\right\lVert^{k+1})
 \label{eqn11},
 \end{aligned}
\end{equation}

where

\begin{equation}
\begin{aligned}
 \boldsymbol{F}_{x_{i_1}x_{i_2}\dots x_{i_k}}\equiv \frac{\partial^k\boldsymbol{F}}{\partial x_{i_1}\partial x_{i_2}\dots \partial x_{i_k}},
 \end{aligned}
  \label{eqn12}
\end{equation}

Since the Jacobian computed at the saddle is diagonalizable, it exists an invertible matrix $\boldsymbol{V}$ such that

\begin{equation}
\nabla \boldsymbol{F}(\boldsymbol{x}^*, \boldsymbol{p}^*) =\boldsymbol{V}\boldsymbol{\Lambda}\boldsymbol{V}^{-1}
 \label{eqn13}
\end{equation}

or,

 \begin{equation}
\boldsymbol{\Lambda}=\boldsymbol{V}^{-1}\nabla \boldsymbol{F}(\boldsymbol{x}^*,\boldsymbol{p}^*)\boldsymbol{V}
 \label{eqn14}
\end{equation}

If all eigenvalues are real then $\boldsymbol{\Lambda}$ is a diagonal matrix:

 \begin{equation}
 \boldsymbol{\Lambda} =\begin{bmatrix}
    \lambda_{1} & 0 & 0 & \dots  & 0 \\
    0 & \lambda_{2} & 0 & \dots  & 0 \\
    \vdots & \vdots & \vdots & \ddots & \vdots \\
    0 & 0 & 0& \dots  & \lambda_{n}
\end{bmatrix}.
 \label{eqn15}
\end{equation}

and $\boldsymbol{V}$ is the matrix whose columns are the eigenvectors $\boldsymbol{v}_j$ of the system's Jacobian  $\nabla \boldsymbol{F}(\boldsymbol{x}^*,\boldsymbol{p}^*)$. If the Jacobian has a complex pair of eigenvalues $\lambda_{k,k+1}=a\pm\beta i$, $\boldsymbol{\Lambda}$ is a block diagonal matrix of the form

 \begin{equation}
 \boldsymbol{\Lambda} =\begin{bmatrix}
    \lambda_{1} & 0 & 0 & \dots  & 0 &0\\
    0 & \lambda_{2} & 0 & \dots  & 0 &0\\
    \vdots & \vdots & \ddots & \vdots & \vdots&\vdots \\
    \boldsymbol{0} & \boldsymbol{0} & \boldsymbol{0}& \boldsymbol{B}_k  & \boldsymbol{0}& \boldsymbol{0}\\
    \vdots & \vdots & \vdots & \dots & \ddots&\vdots \\
    0 & 0 & 0 & \dots  & 0 &\lambda_{n}
\end{bmatrix}.
 \label{eqn16}
\end{equation}

where 

\begin{equation}
 \boldsymbol{B}_k =\begin{bmatrix}
    a&\beta\\
    -\beta& a
\end{bmatrix}.
 \label{eqn17}
\end{equation}

In that case, for the  pair of complex eigenvectors the two corresponding columns of $\boldsymbol{V}$ are assembled by the real and imaginary parts of the complex eigenvector $\boldsymbol{v}_k$.\\
On a saddle, $l$ of these eigenvectors correspond to $l$ eigenvalues with $ |\lambda_{i}|<1$ and $n-l$ of these eigenvectors correspond to $n-l$ eigenvalues with $ |\lambda_i|>1$. 
Let us rearrange the columns of $\boldsymbol{V}$ so that the matrix of eigenvalues $\Lambda$ can be written in a block form as:

\begin{equation}
\boldsymbol{\Lambda}=\left(\begin{array}{cc} \boldsymbol{\Lambda}_s & \boldsymbol{0}\\ \boldsymbol{0} & \boldsymbol{\Lambda}_u \end{array} \right), 
\label{eqn20}
\end{equation}

where $\boldsymbol{\Lambda}_s$ is the $l\times l$ (block) diagonal matrix containing the $l$ eigenvalues with $|\lambda_{i}|<1$ and $\boldsymbol{\Lambda}_u$ is the $(n-l)\times (n-l)$ (block) diagonal matrix containing the $(n-l)$ eigenvalues with $|\lambda_{i}|>1$.


Use the transformation
\begin{equation}
\boldsymbol{x}'=\boldsymbol{V}\boldsymbol{z}, 
\label{eqn21}
\end{equation}

and introduce Eq.(~\ref{eqn21}) in Eq.~(\ref{eqn10}) to get:

\begin{equation}
\boldsymbol{V} \boldsymbol{z}_{k+1}=\nabla \boldsymbol{F}
(\boldsymbol{x}^*,\boldsymbol{p}^*)\boldsymbol{V}\boldsymbol{z}_{k}+\boldsymbol{g}(\boldsymbol{V}\boldsymbol{z}_k,\boldsymbol{x}^*,\boldsymbol{p}^*)
\label{eqn22}
\end{equation}

or

\begin{equation}
 \boldsymbol{z}_{k+1}=\boldsymbol{V}^{-1}\nabla\boldsymbol{F}
(\boldsymbol{x}^*,\boldsymbol{p}^*)\boldsymbol{V}\boldsymbol{z}_{k}+\boldsymbol{V}^{-1} \boldsymbol{g}(\boldsymbol{V}\boldsymbol{z}_k,\boldsymbol{x}^*,\boldsymbol{p}^*)
\label{eqn23}
\end{equation}

Hence, Eq.~(\ref{eqn23}) can be written as:

\begin{equation}
\begin{bmatrix}\boldsymbol{z_s} \\ \boldsymbol{z_u}\end{bmatrix} _{k+1}=
\left(\begin{array}{cc} \boldsymbol{\Lambda}_s & \boldsymbol{0}\\ \boldsymbol{0} & \boldsymbol{\Lambda}_u \end{array} \right) 
\begin{bmatrix} \boldsymbol{z_s}\\ \boldsymbol{z_u}\end{bmatrix} _k+\begin{bmatrix} \boldsymbol{V}_1 & \boldsymbol{V}_2  \end{bmatrix}^{-1}\boldsymbol{g}(\begin{bmatrix} \boldsymbol{V}_1 & \boldsymbol{V}_2  \end{bmatrix}\begin{bmatrix} \boldsymbol{z_s}\\ \boldsymbol{z_u}\end{bmatrix}_k,\boldsymbol{x}^*,\boldsymbol{p}^*),
\label{eqn24}
\end{equation}

where $\boldsymbol{V}_1$ and $\boldsymbol{V}_2$ are the sub-matrices of dimensions $n \times l$ and $n \times n-l$, whose columns contain the eigenvectors corresponding to the eigenvalues inside and outside the unit disc, respectively. Note that $\boldsymbol{z}_s$ and $\boldsymbol{z}_u$ are  uncoupled with respect to the linear terms.\par 
Thus, Eq.~(\ref{eqn24}) can then be re-written as:

\begin{equation}
\boldsymbol{z}_{s,k+1}=\boldsymbol{\Lambda}_s\boldsymbol{z}_{s,k}+\boldsymbol{g}_s(\begin{bmatrix} \boldsymbol{V}_1 & \boldsymbol{V}_2  \end{bmatrix}\begin{bmatrix} \boldsymbol{z_s}\\ \boldsymbol{z_u}\end{bmatrix},\boldsymbol{x}^*,\boldsymbol{p}^*)
\label{eqn25}
\end{equation}
\&
\begin{equation}
\boldsymbol{z}_{u,k+1}=\boldsymbol{\Lambda}_u\boldsymbol{z}_{u,k}+\boldsymbol{g}_u(\begin{bmatrix} \boldsymbol{V}_1 & \boldsymbol{V}_2  \end{bmatrix} \begin{bmatrix} \boldsymbol{z_s}\\ \boldsymbol{z_u}\end{bmatrix},\boldsymbol{x}^*,\boldsymbol{p}^*),
\label{eqn26}
\end{equation}

where
\begin{equation}
\begin{bmatrix} \boldsymbol{g}_s \\ \boldsymbol{g}_u  \end{bmatrix}=\begin{bmatrix} \boldsymbol{V}_1 & \boldsymbol{V}_2  \end{bmatrix}^{-1}\boldsymbol{g}(\boldsymbol{x}',\boldsymbol{x}^*,\boldsymbol{p}^*).
\label{eqn27}
\end{equation}

The fixed point of Eqs.(~\ref{eqn25},\ref{eqn26}) is the ($\boldsymbol{z}_s=0$,$\boldsymbol{z}_u=0$). Thus this implies that

\begin{equation}
\begin{aligned}
\boldsymbol{g}_u(\boldsymbol{0})=\boldsymbol{0}\\
\boldsymbol{g}_s(\boldsymbol{0})=\boldsymbol{0}
\end{aligned}
\label{eqn28}
\end{equation}

Hence, according to the stable manifold theorem (see e.g. \cite{Kuznetsov2004,Wiggins2003})  (A1) and (A2) hold true. Trajectories starting on $\boldsymbol{W}_s$  approach the origin as $k\to\infty$, i.e.: 

\begin{equation}
\forall \boldsymbol{x} \in \boldsymbol{W}_s:{lim_{k\to\infty}\boldsymbol{F}^{k}(\boldsymbol{x},\boldsymbol{p}^*)=\boldsymbol{x}^*}, 
\label{eqn18}
\end{equation}

Trajectories starting on $\boldsymbol{W}_u$  approach the origin as $k\to -\infty$, i.e.: 

\begin{equation}
\forall \boldsymbol{x} \in \boldsymbol{W}_u:{lim_{k\to -\infty}\boldsymbol{F}^{k}(\boldsymbol{x},\boldsymbol{p}^*)=\boldsymbol{x}^*},
\label{eqn19}
\end{equation}

\item (B1),(B2)

By taking Eq.~(\ref{eqn2}), Eq.~(\ref{eqn26}) reads:

\begin{equation}
\boldsymbol{h}_s(\boldsymbol{z}_{s,k+1})=\boldsymbol{\Lambda}_u\boldsymbol{h}_s(\boldsymbol{z}_{s,k})+\boldsymbol{g}_u(\begin{bmatrix} \boldsymbol{V}_1 & \boldsymbol{V}_2  \end{bmatrix} \begin{bmatrix}\boldsymbol{z_s}\\ \boldsymbol{h}_s(\boldsymbol{z}_{s})\end{bmatrix}_k,\boldsymbol{x}^*,\boldsymbol{p}^*) 
\label{eqn29} 
\end{equation}

Then by Eq.~(\ref{eqn29}) and Eq.~(\ref{eqn25}) we obtain:

\begin{equation}
\begin{aligned}
\boldsymbol{h}_s(\boldsymbol{\Lambda}_s\boldsymbol{z}_{s}+\boldsymbol{g}_s(\begin{bmatrix} \boldsymbol{V}_1 & \boldsymbol{V}_2  \end{bmatrix}\begin{bmatrix} \boldsymbol{z_s} \boldsymbol{h}_s(\boldsymbol{z}_{s})\end{bmatrix},\boldsymbol{x}^*,\boldsymbol{p}^*))=\\
\boldsymbol{\Lambda}_u\boldsymbol{h}_s(\boldsymbol{z}_{s})+\boldsymbol{g}_u(\begin{bmatrix} \boldsymbol{V}_1 & \boldsymbol{V}_2  \end{bmatrix}\begin{bmatrix} \boldsymbol{z_s}\\\boldsymbol{h}_s(\boldsymbol{z}_{s})\end{bmatrix},\boldsymbol{x}^*,\boldsymbol{p}^*)) 
\label{eqn30}
\end{aligned}
\end{equation}
with $\boldsymbol{h}_s(\boldsymbol{0})=\boldsymbol{0}$.\\
In a similar manner, it can be shown that the equation given in (B2) holds true on the unstable manifold.
\end{proof}







 \subsection{Parametrization of the Stable and Unstable Manifolds with Truncated Polynomials}
\label{subsec:Parametrization}

As by Theorem 1, the stable and unstable manifolds are smooth non-linear functions of $\boldsymbol{z}_s$ and $\boldsymbol{z}_u$, respectively, then according to the Stone-Weierstrass theorem \cite{Rudin} they can be approximated by any accuracy around $(\boldsymbol{x}^{*},\boldsymbol{p}^{*})$ by a sequence of polynomial functions of $\boldsymbol{z}_s$ and $\boldsymbol{z}_u$, respectively.\par 
For example, for the stable manifold (and similarly for the unstable manifold), $\forall z_{uj}=h_j(\boldsymbol{z}_s), j=1,2,\dots n-l$ we can write:

\begin{equation}
\begin{aligned}
h_j({\boldsymbol{z}_s})=\sum_{k_1=0}^{\infty}\sum_{k_2=0}^{\infty} \dots \sum_{k_l=0}^{\infty} a_{k_1,k_2,\dots,k_l}^{(j)} P_{k_1}(z_{s1}) \dotsm  P_{k_l}(z_{sl}),
\end{aligned}
\label{eqn31}
\end{equation}

where $P_{k_i}$, $i=1,2,..l$ are polynomials (e.g. Chebyshev polynomials) of degree $k_i$.\par
Truncating the series at degree $M$ we get the truncated polynomial approximation:

\begin{equation}
\begin{aligned}
h_j({\boldsymbol{z}_s})\approx \sum_{k_1=0}^{M}\sum_{k_2=0}^{M} \dots \sum_{k_l=0}^{M} a_{k_1,k_2,\dots,k_l}^{(j)} P_{k_1}(z_{s1}) \dotsm  P_{k_l}(z_{sl}).
\end{aligned}
\label{eqn32}
\end{equation}

A simple choice would be to take as polynomials the powers of $\boldsymbol{z}_s$. In that case, Eq.~(\ref{eqn32}) becomes:

\begin{equation}
\begin{aligned}
h_j({\boldsymbol{z}_s})\approx \sum_{k_1=0}^{M}\sum_{k_2=0}^{M} \dots \sum_{k_l=0}^{M} a_{k_1,k_2,\dots,k_l}^{(j)} \prod_{i=1}^{l} z_{si}^{k_i}.
\end{aligned}
\label{eqn33}
\end{equation}

For example if $l=2$, $M=2$, the above expression reads:

\begin{equation}
\begin{aligned}
h_j({\boldsymbol{z}_s})=a^{(j)}_{0,0}z_{s1}^0z_{s2}^0+a^{(j)}_{0,1}z_{s1}^0z_{s2}^1+
a^{(j)}_{0,2}z_{s1}^0z_{s2}^2+
a^{(j)}_{1,0}z_{s1}^1z_{s2}^0+
a^{(j)}_{1,1}z_{s1}^1z_{s2}^1+\\
a^{(j)}_{1,2}z_{s1}^1z_{s2}^2+
a^{(j)}_{2,0}z_{s1}^2z_{s2}^0+
a^{(j)}_{2,1}z_{s1}^2z_{s2}^1+
a^{(j)}_{2,2}z_{s1}^2z_{s2}^2.
\end{aligned}
\label{eqn34}
\end{equation}

The existence of a local analytic solution for the form of nonlinear functional equations Eq.\ref{eqn32} is guaranteed by  the following theorem (see also \cite{Kazantzis2002}:

\begin{theorem}
\cite{Smajdor1968}
Consider the following system of nonlinear functional equations:
\begin{equation}
\boldsymbol{\phi}(\boldsymbol{z})=\boldsymbol{w}(\boldsymbol{z},\boldsymbol{\phi}(\boldsymbol{f}(\boldsymbol{z})),
\label{eqn35}
\end{equation}
where $\boldsymbol{\phi}:R^n \rightarrow R^m$ is an unknown function. Then if:
\begin{enumerate}
\item\ $\boldsymbol{f}:R^n\rightarrow R^n, \boldsymbol{w}:R^n \times R^m \rightarrow R^m$ are analytic functions such that $\boldsymbol{f}(\boldsymbol{0})=\boldsymbol{0}$ and $\boldsymbol{w}(\boldsymbol{0},\boldsymbol{0})=\boldsymbol{0}$
\item The function $\boldsymbol{\phi}$ admits a formal power series solution.
\item  The fixed point that satisfies $\boldsymbol{f}(\boldsymbol{0})=\boldsymbol{0}$ is a hyperbolic point, i.e. none of the eigenvalues of the Jacobian $\nabla_{\boldsymbol{z}}\boldsymbol{f}(\boldsymbol{z=0})$ is on the unit circle.
\end{enumerate}

Then, the above system of nonlinear functional equations admits a unique solution $\boldsymbol{\phi}$ on the form of formal power series which statisy $\boldsymbol{\phi}=\boldsymbol{0}$.

\end{theorem}

Thus, by introducing the polynomial series approximation given by Eq.(~\ref{eqn32}) into Eq.(~\ref{eqn30}) we get $\forall h_{sj}(\boldsymbol{z}_s), j=1,2,\dots n-l$:

\begin{equation}
\begin{aligned}
\sum_{k_1=0}^{M}\sum_{k_2=0}^{M} \dots \sum_{k_l=0}^{M} a_{k_1,k_2,\dots,k_l}^{(j)} P_{k_1}(\hat{z}_{s1}) \dotsm  P_{k_l}(\hat{z}_{sl})=\\
\lambda_{uj} \sum_{k_1=0}^{M}\sum_{k_2=0}^{M} \dots \sum_{k_l=0}^{M} a_{k_1,k_2,\dots,k_l}^{(j)} P_{k_1}({z}_{s1}) \dotsm  P_{k_l}({z}_{sl}) +\\
{g}_{uj}(\begin{bmatrix} \boldsymbol{V}_1 & \boldsymbol{V}_2  \end{bmatrix}\begin{bmatrix} \boldsymbol{z}_{s1}\\
\boldsymbol{z}_{s2}\\
\vdots\\
\boldsymbol{z}_{sl}
\\
\\ \sum_{k_1=0}^{M}\sum_{k_2=0}^{M} \dots \sum_{k_l=0}^{M} a_{k_1,k_2,\dots,k_l}^{(1)} P_{k_1}({z}_{s1}) \dotsm  P_{k_l}({z}_{sl}) \\ \\
\sum_{k_1=0}^{M}\sum_{k_2=0}^{M} \dots \sum_{k_l=0}^{M} a_{k_1,k_2,\dots,k_l}^{(2)} P_{k_1}({z}_{s1}) \dotsm  P_{k_l}({z}_{sl}) \\ \\
\vdots\\
\sum_{k_1=0}^{M}\sum_{k_2=0}^{M} \dots \sum_{k_l=0}^{M} a_{k_1,k_2,\dots,k_l}^{(j)} P_{k_1}({z}_{s1}) \dotsm  P_{k_l}({z}_{sl})\\ \\
\vdots\\
\sum_{k_1=0}^{M}\sum_{k_2=0}^{M} \dots \sum_{k_l=0}^{M} a_{k_1,k_2,\dots,k_l}^{(n-l)} P_{k_1}({z}_{s1}) \dotsm  P_{k_l}({z}_{sl}),
\end{bmatrix},\boldsymbol{x}^*,\boldsymbol{p}^*)),
\end{aligned}
\label{eqn36}
\end{equation}

where $\boldsymbol{\hat{z}}_s=\{\hat{z}_{s1},\dots \hat{z}_{sl}\}$ are nonlinear functions of $\boldsymbol{z}_s=\{{z}_{s1},\dots {z}_{sl}\}$:

\begin{equation}
\boldsymbol{\hat{z}}_s=\boldsymbol{\Lambda}_s\boldsymbol{z}_{s}+\boldsymbol{g}_s(\begin{bmatrix} \boldsymbol{V}_1 & \boldsymbol{V}_2  \end{bmatrix}\begin{bmatrix} \boldsymbol{z_s}\\ \boldsymbol{h}_s(\boldsymbol{z}_{s})\end{bmatrix},\boldsymbol{x}^*,\boldsymbol{p}^*).
\label{eqn37}
\end{equation}

Note that in general, both the left-hand side and the right hand-side of Eq.(~\ref{eqn36})  contain higher order terms than $M$ due to Eq.(~\ref{eqn30}) and the nonlinearities in $g_{uj}$. \\
By equating on both sides of Eq.~\ref{eqn36} the terms up to an order $r<=M$ with respect to $\{{z}_{s1},\dots {z}_{sl}\}$, we get the following  coupled system of nonlinear  equations with respect to the $(n-l)\times (r+1)^l$ polynomial coefficients $a_{k_1,k_2,\dots,k_l}^{(j)},\{j=1,2, \dots n-l\}, \{k_1,k_2,...,k_l=0,1,\dots r\}$:

\begin{equation}
\Psi_{j, i}(a_{k_1,k_2,\dots,k_l}^{(1)},\dots, a_{k_1,k_2,\dots,k_l}^{(n-l)}) = \Phi_{j,i}(a_{k_1,k_2,\dots,k_l}^{(j)})+g_{uj,i}(a_{k_1,k_2,\dots,k_l}^{(1)},\dots, a_{k_1,k_2,\dots,k_l}^{(n-l)}),
\label{eqn38}
\end{equation}

with ${j=1,2,\dots (n-l)},{i=1,2,\dots (r+1)^l}$.

The above system constitutes a nonlinear (in general) system of $(n-l) \times (r+1)^l$ unknowns with $(n-l) \times (r+1)^l$ equations that can be solved iteratively, e.g. using Newton-Raphson.


For example, let us consider the following discrete dynamical system:
\begin{equation}
    \begin{aligned}
    x_1(k+1)= -0.5x_1(k)\\
    x_2(k+1)= -0.5x_2(k)+x_{1}^{2}(k)\\
    x_3(k+1)= 2x_3(k)+x_{2}^{2}(k).
    \end{aligned}
        \label{eqn39}
        \end{equation}

The above system can be written as:
\begin{align}
   \begin{bmatrix} x_1\\x_2\\x_3\end{bmatrix}(k+1) =\begin{bmatrix} -0.5 &0 &0\\0 &-0.5 &0\\0 &0 &2\end{bmatrix}\begin{bmatrix} x_1\\x_2\\x_3\end{bmatrix}(k)+\begin{bmatrix} 0\\x_{1}^{2}(k)\\x_{2}^{2}(k)\end{bmatrix}.
    \label{eqn40}
    \end{align}

\begin{prop}

The stable manifold of the system given by Eq.~(\ref{eqn39}) is given by $h_s(x_1,x_2)=-\frac{4}{7}x_{2}^2+\frac{32}{119}x_{1}^2x_{2} +O(x_{1}^2x_{2}^2)$.
\end{prop}

\begin{proof}
Let us choose a power series expansion up to order two (i.e. $M=2$) of the stable manifold around the fixed point $x_{1}*=x_{2}*=x_{3}*=0$. Hence an approximation of the stable manifold is given by:

\begin{equation}
\begin{aligned}
   x_3=h_s(x_1,x_2)\approx
    a_{0,0}+a_{0,1}x_{2}+
a_{0,2}x_{2}^2+
a_{1,0}x_{1}+
a_{1,1}x_{1}x_{2}+\\
a_{1,2}x_{1}x_{2}^2+
a_{2,0}x_{1}^2+
a_{2,1}x_{1}^2x_{2}+
a_{2,2}x_{1}^2x_{2}^2.
\end{aligned}
    \label{eqn41}
\end{equation}

Here, $\boldsymbol{V}=\begin{bmatrix} 1 &0 &0\\0 &1 &0\\0 &0 &1\end{bmatrix}$. Hence, from Eq.~(\ref{eqn4}) we get:

\begin{equation}
\begin{aligned}
h_s(\begin{bmatrix}
-0.5 &0\\0 &-0.5 \end{bmatrix}\begin{bmatrix}
x_1\\x_2\end{bmatrix}+
\boldsymbol{g}_s(\begin{bmatrix}x_1\\x_2\\ h_s(x_1,x_2)\end{bmatrix})=
2h_s(x_1,x_2)+ 
\boldsymbol{g}_u(\begin{bmatrix}x_1\\x_2\\ h_s(x_1,x_2)\end{bmatrix}) 
\end{aligned}
\label{eqn42}
\end{equation}

or

\begin{equation}
\begin{aligned}
h_s(\begin{bmatrix}
-0.5 &0\\0 &-0.5 \end{bmatrix}\begin{bmatrix}
x_1\\x_2\end{bmatrix}+\
\begin{bmatrix}0\\x_{1}^{2}\end{bmatrix})=
2h_s(x_1,x_2)+x_{2}^{2}
\end{aligned}
\label{eqn43}
\end{equation}

or

\begin{equation}
\begin{aligned}
h_s(\begin{bmatrix}
-0.5x_1\\-0.5x_2+x_{1}^{2}\end{bmatrix})=
2h_s(x_1,x_2)+x_{2}^{2}.
\end{aligned}
\label{eqn44}
\end{equation}

Thus, from Eq.~(\ref{eqn41}) we have:

\begin{equation}
\begin{aligned}
    a_{0,0}+a_{0,1}(-\frac{1}{2}x_2+x_{1}^{2})+
a_{0,2}(-\frac{1}{2}x_2+x_{1}^{2})^2-
\frac{1}{2}a_{1,0}x_{1}-\\
\frac{1}{2}a_{1,1}x_{1}(-\frac{1}{2}x_2+x_{1}^{2})-
\frac{1}{2}a_{1,2}x_{1}(-\frac{1}{2}x_2+x_{1}^{2})^2+\frac{1}{4}a_{2,0}x_{1}^2+\\
\frac{1}{4}a_{2,1}x_{1}^2(-\frac{1}{2}x_2+x_{1}^{2})+
\frac{1}{4}a_{2,2}x_{1}^2(-\frac{1}{2}x_2+x_{1}^{2})^2=\\
2(a_{0,0}+a_{0,1}x_{2}+
a_{0,2}x_{2}^2+
a_{1,0}x_{1}+
a_{1,1}x_{1}x_{2}+
a_{1,2}x_{1}x_{2}^2+\\
a_{2,0}x_{1}^2+
a_{2,1}x_{1}^2x_{2}+
a_{2,2}x_{1}^2x_{2}^2)+x_{2}^{2}
\end{aligned}
    \label{eqn45}
\end{equation}

or

\begin{equation}
\begin{aligned}
    -\frac{1}{2}a_{0,1}x_2+a_{0,1}x_{1}^{2}+
\frac{1}{4}a_{0,2}x_{2}^{2}+a_{0,2}x_{1}^{4}-a_{0,2}x_{2}x_{1}^{2}-
\frac{1}{2}a_{1,0}x_{1}+\\
\frac{1}{4}a_{1,1}x_{1}x_2-\frac{1}{2}a_{1,1}x_{1}^{3}-
\frac{1}{8}a_{1,2}x_{1}x_{2}^{2}-\\
\frac{1}{2}a_{1,2}x_{1}^{5}+\frac{1}{2}a_{1,2}x_2x_{1}^{3}+
\frac{1}{4}a_{2,0}x_{1}^2-
\frac{1}{8}a_{2,1}x_{1}^2x_2+\\
\frac{1}{4}a_{2,1}x_{1}^{4}+
\frac{1}{16}a_{2,2}x_{1}^2x_{2}^{2}+\frac{1}{4}a_{2,2}x_{1}^{6}-\frac{1}{4}a_{2,2}x_2x_{1}^{4}=\\
2a_{0,1}x_{2}+
2a_{0,2}x_{2}^2+
2a_{1,0}x_{1}+
2a_{1,1}x_{1}x_{2}+
2a_{1,2}x_{1}x_{2}^2+\\
2a_{2,0}x_{1}^2+
2a_{2,1}x_{1}^2x_{2}+
2a_{2,2}x_{1}^2x_{2}^2+x_{2}^{2}.
\end{aligned}
    \label{eqn46}
\end{equation}

By equating the coefficients of the corresponding power series up to order two, we get the following system of equations:

\begin{equation}
\begin{aligned}
a_{0,1}=a_{2,0}=a_{1,1}=a_{1,2}=a_{2,2}=0\\
-a_{0,2}-\frac{1}{8}a_{2,1}=2a_{2,1}\\
\frac{1}{4}a_{0,2}=2a_{0,2}+1
\end{aligned}
    \label{eqn47}
\end{equation}

From the above system we get:

\begin{equation}
\begin{aligned}
a_{0,2}=-\frac{4}{7}, a_{2,1}=\frac{32}{119}
\end{aligned}
    \label{eqn48}
\end{equation}

Thus a parametrization of the stable manifold around the saddle point is given by:

\begin{equation}
\begin{aligned}
h_s(x_1,x_2)\approx-\frac{4}{7}x_{2}^2+\frac{32}{119}x_{1}^2x_{2}
\end{aligned}
    \label{eqn49}
\end{equation}

\end{proof}

\begin{prop}
The unstable manifold of the system given by Eq.~\ref{eqn39} is the trivial $x_1=0$, $x_2=0$.
\end{prop}

\begin{proof}

Let us again choose a power series expansion up to order two (i.e. $M=2$) of the unstable manifold around the fixed point $x_{1}^{*}=x_{2}^{*}=x_{3}^{*}=0$. Hence an approximation of the stable manifold is given by:

\begin{equation}
\begin{aligned}
   x_1=h^{(1)}_{u}(x_3)\approx
    a^{(1)}_{0,0}+a^{(1)}_{0,1}x_{3}+
a^{(1)}_{0,2}x_{3}^2
\\
 x_2=h^{(2)}_{u}(x_3)\approx
    a^{(2)}_{0,0}+a^{(2)}_{0,1}x_{3}+
a^{(2)}_{0,2}x_{3}^2.
\end{aligned}
    \label{eqn50}
\end{equation}

Hence, from Eq.~(\ref{eqn5}) we get:

\begin{equation}
\begin{aligned}
\begin{bmatrix}
h^{(1)}_u(
2x_3+{h^{(2)}_{u}(x_3)}^2)\\
h^{(2)}_u(
2x_3+{h^{(2)}_{u}(x_3)}^2)
\end{bmatrix}
= \begin{bmatrix}
    -0.5 & 0\\0 & -0.5
\end{bmatrix}
\begin{bmatrix}
    h^{(1)}_{u}(x_3)\\
    h^{(2)}_{u}(x_3)
\end{bmatrix}
+
\begin{bmatrix}
   0\\{h^{(1)}_{u}(x_3)}^2.
\end{bmatrix}
\end{aligned}
\label{eqn51}
\end{equation}

For the above system of equations it can be easily verified that the unstable manifold is the one with $x_1=0$, $x_2=0$.
\end{proof}

\section{Numerical Approximation of the Stable Manifolds of Microscopic-Stochastic Multiscale and Black-Box Simulators}
\label{sec:MicroscopicSimulators}

Let us assume that due to the complexity of the underlying dynamics evolving across temporal and spatial scales, explicit model equations (such as the ones given by Eq.~(\ref{eqn1})) for the macroscopic (emergent) level are not available in a closed form. Under this hypothesis, we cannot follow the procedure for the analytical approximation of the invariant manifolds as one needs to explicitly know the operator $\boldsymbol{F}$ (i.e. $\boldsymbol{g}_s$ and $\boldsymbol{g}_u$ in Eq.~\ref{eqn6}).

Thus,  when explicit macroscopic equations are not available in a closed form, but a microscopic dynamical simulator is available, the approximation of the invariant manifolds at the macroscopic (the coarse-grained) level requires (a) the bridging of the micro and macro scale, and (b) the numerical approximation of the coarse-grained manifolds. In what follows, we address a new multiscale numerical method for the numerical approximation of the invariant manifolds based on the Equation-Free framework.

Thus, let as assume, that we have a microscopic (such as Brownian dynamics, Monte Carlo, Molecular Dynamics, Agent-based) computational model that, given a microscopic/ detailed distribution of states

\begin{equation}
\boldsymbol{U}_{k} \equiv \boldsymbol{U}(t_{k}) \in R^{N}, N>>1 
\label{eqn52}
\end{equation}

at time $t_k=kT_U$, will report the values of the evolved microscopic/detailed distribution after a time horizon $T_U$:

\begin{equation}
\boldsymbol{U}_{k+1} =\boldsymbol{\Phi}_{T_{U}}(\boldsymbol{U}_k,\boldsymbol{p}),  \label{eqn53}
\end{equation}

$\boldsymbol{\Phi}_{T_{U}}:R^N \times R^m \rightarrow R^N$ is the time-evolution microscopic operator, $\boldsymbol{p} \in R^m$ is the vector of the complex system parameters.\\

A basic assumption underlying the concept of Equation-Free numerical framework is that after some time $t>>T_U$ the emergent coarse-grained dynamics are governed by a few variables, say, $\boldsymbol{x} \in R^n, n<<N$. Usually these ``few" observables are the first few moments of the underlying microscopic distribution. This implies that there is a slow coarse-grained manifold that can be parametrized by $\boldsymbol{x}$. The assumption of the existence of a slow coarse-grained manifold asserts that the higher order moments of the microscopic distribution, say, $\boldsymbol{y} \in R^{N-n}$, of the microscopic distribution $\boldsymbol{U}$  become, relatively fast over time, functionals of the $n$ lower-order moments of the microscopic distribution described by the vector $\boldsymbol{x}$. This dependence can be described at the moments-space as a singularly perturbed system of the form:

\begin{equation}
\begin{aligned}
\boldsymbol{x}_{k+1}= \boldsymbol{X}(\boldsymbol{x}_k,\boldsymbol{y}_k,\boldsymbol{p},\epsilon)\\
\epsilon \boldsymbol{y}_{k+1}= \boldsymbol{Y}(\boldsymbol{x}_k,\boldsymbol{y}_k,\boldsymbol{p},\epsilon),
\end{aligned}
\label{eqn54}
\end{equation}

where $\epsilon>0$ is a sufficiently small number. Under the above description and assumptions the following Theorem can be proved.

\begin{theorem}
[Fenichel's Theorem \cite{Fenichel1979}]. Let us assume that the functions $\boldsymbol{X}:R^n \times R^{N-n} \times R^m \rightarrow R^n$, $\boldsymbol{Y}:R^n \times R^{N-n} \times R^m \rightarrow R^{N-n}$ $\in C^{r}, r<\infty$ in an open set around a hyperbolic fixed point. Then the dynamics of the system given by Eq.~\ref{eqn54} can be reduced to:

\begin{equation}
\boldsymbol{x}_{k+1}= \boldsymbol{X}(\boldsymbol{x}_k,\boldsymbol{\chi}(\boldsymbol{x}_k,\boldsymbol{p},\epsilon),\boldsymbol{p})\label{eqn55}
\end{equation}

on a smooth manifold defined by:

\begin{equation}
M_{\epsilon}=\{(\boldsymbol{x},\boldsymbol{y}) \in R^{n} \times R^{N-n}:\boldsymbol{y}= \boldsymbol{\chi}(\boldsymbol{x},\boldsymbol{p},\epsilon) \}
\label{eqn56}
\end{equation}

The manifold $M_{\epsilon}$ is diffeomorphic and $O(\epsilon)$ close to the $M_{0}$ manifold defined for $\epsilon=0$. Moreover, the manifold $M_{\epsilon}$ is locally invariant under the dynamics given by Eq.~(\ref{eqn54}).

\end{theorem}

$M_{\epsilon}$  defines the ``slow" manifold on which the dynamics of the system evolve after a short (in the macroscopic scale) time horizon.\\

 Under this perspective and under the assumptions of the Fenichel's theorem \cite{Fenichel1979} let us define the coarse-grained map:

\begin{equation}
\boldsymbol{x}_{k+1} =\boldsymbol{F}_T(\boldsymbol{x}_k, \boldsymbol{p}),
\label{eqn57}
\end{equation}

where $\boldsymbol{F}_T:R^n \times R^m \rightarrow R^n$ is a smooth multivariable, vector-valued function having $\boldsymbol{x}_k$ as initial condition and $T>>T_U$.

The above coarse-grained map which describes the system dynamics \emph{on} the slow coarse-grained manifold $M_{\epsilon}$ can be obtained by finding $\boldsymbol{\chi}$ that relates the higher order moments of the microscopic distribution $\boldsymbol{U}_{k}$ to the lower order moments $\boldsymbol{x}$. 

The Equation-free approach through the concept of the coarse timestepper bypasses the need to extract such a relation analytically which in most of the cases is an ``overwhelming" difficult task and can introduce modelling biases (see the critical discussion in \cite{Reppas2012}). The Equation-free approach provides such relations in a numerical way ``on demand": relatively short calls of the detailed simulator provide this closure (refer to \cite{Kevrekidis2003,Kevrekidis2004,Siettos2003} for more detailed discussions). Briefly, the coarse timestepper consists of the following basic steps:

Given the set of the macroscopic variables at time $t_0$:

\indent (a) Prescribe the coarse-grained initial conditions $\boldsymbol{x}(t_0) \equiv \boldsymbol{x}_0$.\\
\indent (b) Transform them through a lifting operator $\boldsymbol{\mu}$ to consistent microscopic distributions $\boldsymbol{U}(t_0)=\boldsymbol{\mu} \boldsymbol{x}(t_0)$.\\
\indent (c) Evolve these distributions in time using the microscopic/detailed simulator for a short macroscopic time $T$ to get $\boldsymbol{ U}(t_0+T)$.
The choice of $T$ is associated with the (estimated) spectral gap of the linearization of the unavailable closed macroscopic equations.\\
\indent (d) Obtain again the values of the coarse-grained variables using a restriction operator $\boldsymbol{M}$: $\boldsymbol{x}_{k+1} \equiv \boldsymbol {x}(t_0+T)=\boldsymbol{M} \boldsymbol{U}(t_0+T)$.\\
\\
The above steps, constitute the black box coarse timestepper, that, given an initial coarse-grained state of the system $\{\boldsymbol{x}_k,\boldsymbol{p}\}$, at time $t_k$ will report the result of the integration of the microscopic rules after a given time-horizon $T$
(at time $t_{k+1}$), i.e. $\boldsymbol{x}_{k+1} = \boldsymbol{F}_T(\boldsymbol{x}_k,\boldsymbol{p})$.

Now one can ``wrap" around the coarse timestepper (given by Eq.(\ref{eqn57})), numerical methods such as the Newton-Raphson method (for low-order systems) to converge to coarse-grained fixed points and investigate their stability. For large-scale systems one can also employ matrix-free methods
such as Newton-GMRES \cite{Kelley1995} to find the coarse-grained fixed points and Arnoldi iterative algorithms \cite{Saad2011} to estimate the dominant eigenvalues of the coarse linearization, that dictate the stability of the coarse-grained fixed points of the unavailable
macroscopic evolution equations.\\

The coarse-grained Jacobian $\nabla \boldsymbol{F}_T(\boldsymbol{x}^*, \boldsymbol{p}^*)$ can be computed by appropriately perturbing the coarse-grained initial conditions fed to the coarse timestepper (3). For low to medium dimensions the $i-th$ column of the Jacobian matrix can be evaluated numerically as

\begin{equation}
 \nabla_{x_i} \boldsymbol{F}_T(x_i, \boldsymbol{p}) \approx \frac{\boldsymbol{F}_T(\boldsymbol{x}+\epsilon \boldsymbol{e}_i, \boldsymbol{p})-\boldsymbol{F}_T(\boldsymbol{x},\boldsymbol{p})} {\epsilon},
 \label{eqn58}
\end{equation}

 where $\boldsymbol{e}_i$ is the unit vector with one at the $i-th$ component and zero in all other components.\par
 Then one can solve the eigenvalue problem
 
 \begin{equation}
 \nabla \boldsymbol{F}_T(\boldsymbol{x}^*, \boldsymbol{p}^*) \boldsymbol{v}_{j}=\boldsymbol{\lambda}_{j}\boldsymbol{v}_{j}
 \label{eqn59}
\end{equation}

with direct solvers. 

The continuation of solutions branches around turning points can be achieved by standard continuation techniques such as the pseudo-arc-length continuation \cite{keller1977numerical}. For example, given two already computed stable fixed points points $(\boldsymbol{x}^{(0)}, \boldsymbol{p}^{(0)})$ and $(\boldsymbol{x}^{(1)}, \boldsymbol{p}^{(1)})$, convergence on saddle fixed points can be achieved by one-dimensional parameter (say ${p}_i$, the i-th element of the $\boldsymbol{p}$) continuation past turning points. This  procedure involves the iterative solution of the following linearized system:

\begin{equation}
\begin{bmatrix}
\displaystyle
    I-\nabla \boldsymbol{F}_T(\boldsymbol{x}, \boldsymbol{p}) &  \nabla_{p_{i}} \boldsymbol{F}_T(\boldsymbol{x}, \boldsymbol{p})  \\
   \frac{(\boldsymbol{x}^{(1)} - \boldsymbol{x}^{(0)})'}{\Delta s}  &  \frac{p^{(1)}_{i} - p^{(0)}_i}{\Delta s}
\end{bmatrix}
\begin{bmatrix}
    dx   \\
    dp 
\end{bmatrix}
= -
\begin{bmatrix}
    \boldsymbol{F}_T(\boldsymbol{x}, \boldsymbol{p})   \\
    N(\boldsymbol{x}, \boldsymbol{p})
\end{bmatrix}
\label{eqn60}
\end{equation}

where

\begin{equation}
N(\boldsymbol{x}, \boldsymbol{p}) = \frac{(\boldsymbol{x}^{(1)} - \boldsymbol{x}^{(0)})'}{\Delta s} (\boldsymbol{x} - \boldsymbol{x}^{(1)}) + \frac{p^{(1)}_{i} - p^{(0)}_i}{\Delta s} (p_i- p^{(1)}_{i}) - \Delta s = 0
\label{eqn61}
\end{equation}

is the pseudo-arc-length condition; $\Delta s$ is the continuation step.  Eq.(~\ref{eqn61}) constrains the fixed point $(x^*, p^*)$ that is computed iteratively by Eq.(\ref{eqn60}) to lie on a hyperplane perpendicular to the tangent of the bifurcation diagram at $(\boldsymbol{x}^{(1)}, p^{(1)}_{i})$ at a distance $\Delta s$ from it.
For the above procedure to be accurate, one should perform the required computations when the system lies on the slow manifold. If the gap between the fast and slow time scales is very big then the time required for trajectories starting off the slow manifold to reach the slow manifold will be very small compared to $T$; hence the coarse-grained computations will not be affected for any practical means. Nevertheless, one can enhance the computing accuracy by forcing the system to start on the slow manifold (using for example the algorithms presented in \cite{Siettos_2011}, \cite{GearKevrekidis}, \cite{GearKaper}).\\

Returning back to the problem of numerical approximation of the stable manifold, as now there are no analytical expressions for the right-hand side of the evolution equations the condition for the derivation of the stable manifold: 

\begin{equation}
\begin{aligned}
\boldsymbol{h}_s(\boldsymbol{\Lambda}_s\boldsymbol{z}_{s}+\boldsymbol{g}_s(\begin{bmatrix} \boldsymbol{V}_1 & \boldsymbol{V}_2  \end{bmatrix}\begin{bmatrix} \boldsymbol{z_s} \boldsymbol{h}_s(\boldsymbol{z}_{s})\end{bmatrix},\boldsymbol{x}^*,\boldsymbol{p}^*))=\\
\boldsymbol{\Lambda}_u\boldsymbol{h}_s(\boldsymbol{z}_{s})+\boldsymbol{g}_u(\begin{bmatrix} \boldsymbol{V}_1 & \boldsymbol{V}_2  \end{bmatrix}\begin{bmatrix} \boldsymbol{z_s}\\\boldsymbol{h}_s(\boldsymbol{z}_{s})\end{bmatrix},\boldsymbol{x}^*,\boldsymbol{p}^*)) 
\end{aligned}
\label{eqn62}
\end{equation}

has to be solved numerically. In general, due to the nonlinear dependence of $\boldsymbol{g}_s$ and $\boldsymbol{g}_u$ on $\boldsymbol{z}_s$ (and $\boldsymbol{z}_u=\boldsymbol{h}(\boldsymbol{z}_s)$) the parameter estimations of the polynomial coefficients becomes a non-linear optimization problem. Thus one can try to find the coefficients $\boldsymbol{a}$ of the polynomial approximation of the stable $\boldsymbol{z}_u=\boldsymbol{h}(\boldsymbol{z}_s)$ by minimizing the nonlinear objective function with respect to the vector of the unknown polynomial coefficients, say $\boldsymbol{q}$:

\begin{equation}
J(\boldsymbol{q})=\arg_{\boldsymbol{q}} \min {\left\lVert \boldsymbol{r}(\boldsymbol{q})\right\rVert}_{2}^{2},
\label{eqn63}
\end{equation}

where,

\begin{equation}
\begin{aligned}
\boldsymbol{r}(\boldsymbol{q})=\boldsymbol{h}_s(\boldsymbol{\Lambda}_s\boldsymbol{z}_{s}+\boldsymbol{g}_s(\begin{bmatrix} \boldsymbol{V}_1 & \boldsymbol{V}_2  \end{bmatrix}\begin{bmatrix} \boldsymbol{z_s} \boldsymbol{h}_s(\boldsymbol{z}_{s})\end{bmatrix},\boldsymbol{x}^*,\boldsymbol{p}^*))-\\
\boldsymbol{\Lambda}_u\boldsymbol{h}_s(\boldsymbol{z}_{s})+\boldsymbol{g}_u(\begin{bmatrix} \boldsymbol{V}_1 & \boldsymbol{V}_2  \end{bmatrix}\begin{bmatrix} \boldsymbol{z_s}\\\boldsymbol{h}_s(\boldsymbol{z}_{s})\end{bmatrix},\boldsymbol{x}^*,\boldsymbol{p}^*)).
\label{eqn64}
\end{aligned}
\end{equation}

The above constitutes a non-linear least-squares problem which can be solved numerically through the concept of coarse-timestepper of the microscopic simulator  with an iterative algorithm such as  the Newton-Raphson algorithm as described in the following steps:

\begin{enumerate}

\item Construct the coarse-timestepper given by the map (\ref{eqn57}) using appropriate lifting $\boldsymbol{\mu}$ and restricting $\boldsymbol{M}$ operators of the microscopic evolved distributions.

\item ``Wrap" around the coarse-timestepper a continuation technique (e.g. the pseudo-arc-length continuation) to converge to a saddle fixed point ($\boldsymbol{x}^{*}, \boldsymbol{p}^{*}$).

\item Compute the coarse-grained Jacobian  $\nabla \boldsymbol{F}_T(\boldsymbol{x}^*, \boldsymbol{p}^*)$ and solve the eigenvalue problem $\nabla \boldsymbol{F}_T(\boldsymbol{x}^*, \boldsymbol{p}^*) \boldsymbol{V}=\boldsymbol{\Lambda}\boldsymbol{V}$. Find the $l$ stable and $n-l$ unstable eignemodes. Rearrange $\boldsymbol{V}$ as $\boldsymbol{V}=\begin{bmatrix}
    \boldsymbol{V}_1 & \boldsymbol{V}_2\end{bmatrix}$ with $\boldsymbol{V}_1$ being the $n \times l$ matrix whose columns are the eigevectors of the Jacobian that correspond to the $l$ eigenvalues lying inside the unit circle, $\boldsymbol{V}_2$ is a $n \times n-l$ matrix whose columns are the eigevectors of the Jacobian that correspond to the $n-l$ eigenvalues lying outside the unit circle.

\item Choose a certain set of polynomials as well as their maximum order $M$ for the numerical approximation of the $j$-th element, say $h_{js}$ of the stable manifold in the form of:

\begin{equation}
\begin{aligned}
z_{ju}(\boldsymbol{z}_s)\equiv h_{js}(\boldsymbol{z}_s)=\\ \sum_{k_1=0}^{M}\sum_{k_2=0}^{M} \dots \sum_{k_l=0}^{M} a_{k_1,k_2,\dots,k_l}^{(j)} P_{k_1}(z_{1s}) \dotsm  P_{k_l}(z_{ls}),\\
j=1,2,\dots n-l
\label{eqn65}
\end{aligned}
\end{equation}

where the variables $\boldsymbol{z}_s, \boldsymbol{z}_u$ are defined by the transformation

\begin{equation}
\begin{aligned}
\boldsymbol{x}'=\begin{bmatrix} \boldsymbol{V}_1 & \boldsymbol{V}_2  \end{bmatrix} \begin{bmatrix} \boldsymbol{z}_s \\ \boldsymbol{z}_u  \end{bmatrix},  \boldsymbol{x}'=\boldsymbol{x}-\boldsymbol{x}^*;  
\label{eqn66}
\end{aligned}
\end{equation}

\item Denote with $\boldsymbol{q}^{(j)}$ the vector with the unknown polynomial coefficients $a_{k_1,k_2,\dots,k_l}^{(j)}$, $j=1,2,\dots n-l$, $k_{1,2}=0,1,\dots M$ of the $j$-th element ($h_{js}$) of the stable manifold. Set an initial guess for $\boldsymbol{q}$.

\item Select $n_p$ points $\boldsymbol{x}_i=1,2,\dots n_p$ within a certain distance $B$ around ($\boldsymbol{x}^{*}, \boldsymbol{p}^{*}$) where an approximation of the stable manifold is sought, and at a certain distance from it, i.e. $ \epsilon_d<\left\lVert \boldsymbol{x}_i -\boldsymbol{x}^{*}\right\rVert <B$.

\item \textbf{Use the coarse-timestepper to construct the map:}

\begin{equation}
\boldsymbol{q}^{(j),(r+1)}=\boldsymbol{Q^{(j)}}(\boldsymbol{q}^{(j),(r)})
\label{eqn67}
\end{equation}

For each of the $\boldsymbol{x}_i=1,2,\dots n_p$:
\begin{itemize}
\item Set $k=0$. 
\item \textbf{For $k=0,1,2,\dots k_{max}$}
\begin{itemize}
    \item Use the transformation (~\ref{eqn66}) to find $\boldsymbol{z}_{s,k}$. Given $\boldsymbol{q}^{r}$, constrain $\boldsymbol{z}_{u,k}$ on the stable manifold, using  Eq.~(\ref{eqn65}).
    \item Use the transformation (~\ref{eqn66}) to find back $\boldsymbol{x}_k$ based on $\boldsymbol{z}_{s,k}, \boldsymbol{z}_{u,k}= (h{z}_{s,k})$.
    \item Use the coarse-timestepper  (~\ref{eqn57}) to find $\boldsymbol{x}_{k+1}$.
    \item Use the transformation (~\ref{eqn66}) to find $\boldsymbol{z}_{s,k+1}$.
    \item Use the truncated polynomial approximation given by Eq.(~\ref{eqn65}) evaluated at $\boldsymbol{q}^{r}$ to find $\boldsymbol{z}_{u,k+1}$
\end{itemize}
\item \textbf{End For}
\item $\forall\boldsymbol{x}_i=1,2,\dots n_p$, and $\forall \boldsymbol{z}_{s,k}, k=0,1,\dots k_{max}$ construct the matrix $\boldsymbol{A}$, whose columns contain the values of each one of the polynomials $P_{k_1}(z_{1s}) \dotsm  P_{k_l}(z_{ls})$, $k_1,k_2,\dots k_m =1,2,\dots M$.\\
For example if one chooses a power series expansion with  $l=2$ (i.e. $\boldsymbol{z}_s \in R^{2})$, the matrix $\boldsymbol{A}$ is of the following form:

\begin{equation}
\boldsymbol{A}=\begin{bmatrix}
   z^{1}_{s2,0} & {z^{1}_{s2,0}}^2 & {z^{1}_{s2,0}}^2 & z^{1}_{s1,0} & \dots & {z^{1}_{s1,0}}^2{z^{1}_{s2,0}}^2\\
   \\
    z^{1}_{s2,1} & {z^{1}_{s2,1}}^2 & {z^{1}_{s2,1}}^2 & z^{1}_{s1,1} & \dots & {z^{1}_{s1,1}}^2{z^{1}_{s2,1}}^2\\
    \vdots & \vdots & \vdots  & \vdots &\vdots &\vdots\\
   z^{1}_{s2,k_{max}} & {z^{1}_{s2,k_{max}}}^2 & {z^{1}_{s2,k_{max}}}^2 & z^{1}_{s1,k_{max}} & \dots & {z^{1}_{s1,k_{max}}}^2{z^{1}_{s2,k_{max}}}^2\\
    \\
    z^{2}_{s2,0} & {z^{2}_{s2,0}}^2 & {z^{2}_{s2,0}}^2 & z^{2}_{s1,0} & \dots & {z^{2}_{s1,0}}^2{z^{2}_{s2,0}}^2\\
    \vdots & \vdots & \vdots  & \vdots &\vdots &\vdots \\
    z^{2}_{s2,k_{max}} & {z^{2}_{s2,k_{max}}}^2 & {z^{2}_{s2,k_{max}}}^2 & z^{2}_{s1,k_{max}} & \dots & {z^{2}_{s1,k_{max}}}^2{z^{2}_{s2,k_{max}}}^2\\
    \vdots & \vdots & \vdots  & \vdots &\vdots &\vdots\\
   z^{np}_{s2,k_{max}} & {z^{np}_{s2,k_{max}}}^2 & {z^{np}_{s2,k_{max}}}^2 & z^{np}_{s1,k_{max}} & \dots & {z^{np}_{s1,k_{max}}}^2{z^{np}_{s2,k_{max}}}^2\\
\end{bmatrix}
\label{eqn68}
\end{equation}

where $z^{i}_{sj,k}$ denotes the $j$-th element of the vector $\boldsymbol{z}_s$ in the $k$-th time step resulting by the $i$-th $\boldsymbol{x}_i$ point. The number of the time steps $k_{max}$ has to be chosen so that the number of rows is greater than the number of columns (i.e. the number of polynomials that are used for the approximation).

\item Find $\boldsymbol{q}^{(r+1)}$ by solving the linear least squares problem

\begin{equation}
\arg\min_{\boldsymbol{q}^{(j),(r+1)}} \left\lVert \boldsymbol{A}\boldsymbol{q}^{(j),(r+1)}-\boldsymbol{b}\right\rVert 
\label{eqn69}
\end{equation}

where $\boldsymbol{b}=\begin{bmatrix}
     z^{(1)}_{uj,0} & z^{(1)}_{uj,1} &\dots z^{(1)}_{uj,k_{max}}  &\dots z^{(np)}_{uj,0} & z^{(np)}_{uj,1} &\dots z^{(np)}_{uj,k_{max}}
\end{bmatrix}'$

The optimal solution of the above linear least-squares problem is given by the solution of 
\begin{equation}
 \boldsymbol{A}'\boldsymbol{A}\boldsymbol{q}^{(j),(r+1)}=\boldsymbol{A}'\boldsymbol{b}
 \label{eqn70}
\end{equation}

If the matrix $\boldsymbol{A}'\boldsymbol{A}$ is of full rank, then the above system has a unique solution given by:

\begin{equation}
 \boldsymbol{q}^{(j),(r+1)}=(\boldsymbol{A}'\boldsymbol{A})^{-1}\boldsymbol{A}'\boldsymbol{b}
 \label{eqn71}
\end{equation}

Note that if the the initial points $\boldsymbol{x}_i$ are chosen close enough to the fixed point $\boldsymbol{x}^*$ and/or the number of time-steps $k_{max}$ are relatively large then as $\boldsymbol{z}_s \rightarrow 0$ the matrix $\boldsymbol{A}'\boldsymbol{A}$ will not be of full rank as the higher order terms of the polynomials expansion will approach fast zero. In that case one could use the Moore-Penrose pseudoinverse of $\boldsymbol{A}'\boldsymbol{A}$ to solve (~\ref{eqn70}) and the solution reads:

\begin{equation}
 \boldsymbol{q}^{{+}^{(j),(r+1)}}=A^{+}\boldsymbol{b}
 \label{eqn72}
\end{equation}

where the pseudo-inverse matrix $A^{+}$ that is obtain by the Singular Value Decomposition (SVD) of the matrix $\boldsymbol{A}$:

\begin{equation}
 \boldsymbol{A}^{+}=\boldsymbol{V}\boldsymbol{\Sigma}^{+}\boldsymbol{U}',
 \label{eqn73}
\end{equation}

where, $\Sigma^{+}$ is the inverse of sub-block diagonal matrix containing the non-zero singular values of the SVD decomposition of $\boldsymbol{A}$.

\end{itemize}

\item \textbf{Find the polynomials coefficients through a fixed-iteration algorithm (e.g. Newton-Raphson) around the map given by Eq.~\ref{eqn50}} as constructed in step 7.

\begin{itemize}

\item Set convergence tolerance, $tol$ for the approximation of the polynomial coefficients. Set $r=0$ and define $\boldsymbol{dq}^{(r)}=\left\lVert \boldsymbol{q}^{(j),(r+1)} -\boldsymbol{q}^{(j),(r)}\right\rVert$

\item \textbf{Do\ while $d^{(r)}$>$tol$}
\begin{itemize}
    \item Compute $\boldsymbol{q}^{(j),(r)}$
    \item Use the coarse-timestepper as in \emph{Step 7} to compute $\boldsymbol{q}^{(j),(r+1)}=\boldsymbol{Q}^{(j)}(\boldsymbol{q}^{(j),(r)})$
    \item Set $\boldsymbol{f}^{(r)}=\boldsymbol{q}^{(j),(r)}-\boldsymbol{q}^{(j),(r+1)}$
    \item Compute the Jacobian $\nabla \boldsymbol{Q}^{(j)}(\boldsymbol{q}^{(j),(r)})$ by perturbing appropriately $\boldsymbol{q}^{(j),(r)}$
    \item Solve the system 
    
    \begin{equation}
       \begin{bmatrix}
           \boldsymbol{I} -\nabla \boldsymbol{Q}^{(j)}(\boldsymbol{q}^{(j),(r)})\end{bmatrix}\boldsymbol{dq}^{(r)}=-\boldsymbol{f}^{(r)}
 \label{eqn74}
   \end{equation}
   
 to get $\boldsymbol{dq}^{(r)}$
  \item Update the solution: set $r=r+1$ and compute the new estimation for the polynomial coefficients:
  
  \begin{equation}
  \boldsymbol{q}^{(j),(r)}=\boldsymbol{q}^{(j),(r-1)}+\boldsymbol{dq}^{(r)}
   \label{eqn75}
\end{equation}
    
\end{itemize}
\item \textbf{End Do\ while}
\end{itemize}

\end{enumerate}

\section{The Illustrative Examples: Numerical Results}
\label{sec:NumericalResults}

The proposed approach is illustrated through two examples: (a) the toy model ~(\ref{eqn39}) and (b) a Monte Carlo simulation of a catalytic reaction on a lattice, for which we have also derived analytically an approximation of the stable and unstable manifolds based on the mean field model. 
 
\subsection{The Toy Model}
\label{subsec:toy}

 In Proposition 1, we showed that the stable manifold of the discrete time model given by Eq.~\ref{eqn39} is given by 

\begin{equation}
\begin{aligned}
h_s(x_1,x_2)\approx-\frac{4}{7}x_{2}^2+\frac{32}{119}x_{1}^2x_{2}
\end{aligned}
    \label{eqn76}
\end{equation}

Here, we will derive a numerical approximation of the stable manifold by assuming that the equations of the model are not explicitly known. Our assumption is that we have a black-box model that given initial conditions $(\boldsymbol{x}_1(0), \boldsymbol{x}_2(0), \boldsymbol{x}_3(0))$ it outputs $(\boldsymbol{x}_1(k), \boldsymbol{x}_2(k), \boldsymbol{x}_3(k))$, 
$\{k=1,2,\dots\}$.
The saddle point is the $(\boldsymbol{x}_1\boldsymbol{x}_2, \boldsymbol{x}_3)=(0,0,0)$. The Jacobian on the saddle is approximated by central finite differences with $\epsilon=0.01$ as perturbation on the initial conditions and running the simulator for one step $k=1$. By doing so, the numerical approximation of the Jacobian actually coincides for any practical means with the analytical one. The eigenvalues are $\lambda_1=-0.5$, $\lambda_2=-0.5$, $\lambda_3=2$ and the eigenvectors are given by $\boldsymbol{e}_i$, i.e. the unit vectors with one at the $i-th$ component and zero in all other components. From the above, it is clear that $z_{1s}=x_1$, $z_{2s}=x_2$, $z_{1u}=x_3$. Thus, we chose a power series expansion of the manifold around the saddle as 

\begin{equation}
\begin{aligned}
   x_3=h_s(x_1,x_2)\approx
   a_{0,1}x_{2}+
a_{0,2}x_{2}^2+
a_{1,0}x_{1}+
a_{1,1}x_{1}x_{2}+
a_{1,2}x_{1}x_{2}^2+\\
a_{2,0}x_{1}^2+
a_{2,1}x_{1}^2x_{2}+
a_{2,2}x_{1}^2x_{2}^2
    \label{eqn77}
    \end{aligned}
\end{equation}

For the construction of the map (see Eq.~\ref{eqn67}) we have used the following parameters: $kmax=3$, $np=4$, $\boldsymbol{z}_{s1}=(-0.2,-0.2)$, $\boldsymbol{z}_{s2}=(-0.2,0.2)$, $\boldsymbol{z}_{s3}=(0.2,-0.2)$, $\boldsymbol{z}_{s4}=(0.2,0.2)$, and central finite differences with $\epsilon=0.05$, for the numerical approximation of the Jacobian $\nabla \boldsymbol{Q}$ that is required for the Newton-Raphson iterations; the tolerance was set to $tol=1E-04$, and the initial guess of the power expansion coefficients was set as $\boldsymbol{q}^{0} \equiv$
($a_{0,1}$,$a_{0,2}$,$a_{1,0}$,$a_{1,1}$,$a_{1,2}$,$a_{2,0}$,$a_{2,1}$,$a_{2,2}$)=( 0.1,0.25,-0.3,-0.15,-0.15,0.1,-0.1,0.15).\par
The Newton-Raphson iterations are as follows:

\begin{table}[h!]
    \centering
\begin{tabular}{c c c c c c c c c }
error &$a_{0,1}$&$a_{0,2}$&$a_{1,0}$&$a_{1,1}$&$a_{1,2}$&$a_{2,0}$&$a_{2,1}$&$a_{2,2}$\\
\hline
$1.026$& $1.41E^{-06}$ &$-0.571$&$-9.06E^{-13}$&$3.91E^{-13}$&$1.13E^{-10}$& $-6E^{-04}$&$0.268$&$-0.26$\\
$6.49E^{-09}$&$1.41E^{-06}$&$-0.571$&$-4.77E^{-15}$&$2.12E^{-15}$&$5.93E^{-13}$&$-6E^{-04}$&$0.268$&$-0.26$\\
\hline
\end{tabular}
\caption{Newton-Raphson iterations for the numerical approximation of the stable manifold of the black-box simulator whose model is given by Eq.~\ref{eqn39}. The initial guess of the power expansion coefficients was set as 
($a_{0,1}$,$a_{0,2}$,$a_{1,0}$,$a_{1,1}$,$a_{1,2}$,$a_{2,0}$,$a_{2,1}$,$a_{2,2}$)=(0.1,0.25,-0.3,-0.15,-0.15,0.1,-0.1,0.15)
}.
\label{Table1}
\end{table}

Thus, the numerical approximation of the stable manifold as derived by the proposed numerical algorithm reads:

\begin{equation}
h_s(x_1,x_2)\approx-0.5708x_{2}^2+0.2687x_{1}^2x_{2}-0.2598x_{1}^2x_{2}^2.
\label{eqn78}
\end{equation}

A comparison with the analytical approximation above shows that the approximation error for $a_{0,2}$ is about $6E^{-4}$ and for $a_{2,1}$ is about $1E^{-4}$. The numerical scheme outputs also a non-zero coefficient for $a_{2,2}$ which is not present in the analytical approximation. This is due to the truncation of the power expansion to second order terms: when equating the terms on both sides of Eq.~\ref{eqn46} higher order powers than three are set to zero. One can confirm the contribution of this extra term found by the numerical scheme by simple simulations.
For example, by setting as initial conditions $x_1(0)=0.2$, $x_2(0)=0.2$ and $x_3(0)=-\frac{4}{7}x_{2}^2+\frac{32}{119}x_{1}^2x_{2}$ we get the results shown in Table~\ref{Table2}.

\begin{table}[h!]
    \centering
\begin{tabular}{c|c c c c c c} 
$k$&0&1&2&3&4&5\\
\hline
$x_1(k)$ &$0.2$ & $-0.1$&$0.05$ &$-0.025$& $0.0125$&$-0.00625$ \\
$x_2(k)$&$0.2$&$-0.06$ &$0.04$&$-0.0175$&$0.00938$&$-0.00453$\\
$x_3(k)$&$-0.020705$&$-0.00141$&$-7.76E^{-04}$& $0.003152$&$0.006612$&$0.01331$\\
\hline
\end{tabular}
\\
  \caption{Numerical simulation of the model given by Eq.~\ref{eqn39} setting as initial conditions on the manifold approximated by  $x_3(0)=-\frac{4}{7}x_{2}^2+\frac{32}{119}x_{1}^2x_{2}$; the other initial conditions were set to $x_1(0)=0.2$, $x_2(0)=0.2$}
    \label{Table2}
\end{table}

Note that $x_3(k)$ goes to zero and then after $k=3$ it diverges due to the (truncated) approximation of the manifold. \\ If we add the extra term found with the numerical scheme, and start with the same initial conditions for $x_1(0)$ and $x_2(0)$, but with $x_3(0)=-\frac{4}{7}x_{2}^2+\frac{32}{119}x_{1}^2x_{2}-0.2598x_{1}^2x_{2}^2$ we get the results shown in Table~\ref{Table3}.

\begin{table}[h!]
    \centering
\begin{tabular}{c|c c c c c c} 
$k$&0&1&2&3&4&5\\
\hline
$x_1(k)$ &$0.2$ & $-0.1$&$0.05$ &$-0.025$& $0.0125$&$-0.00625$ \\
$x_2(k)$&$0.2$&$-0.06$ &$0.04$&$-0.0175$&$0.00937$&$-0.00453$\\
$x_3(k)$&$-0.02112$&$-2.243E^{-03}$&$-8.86E^{-04}$& $-1.72E^{-04}$&$-3.87E^{-05}$&$1.0395E^{-05}$\\
\hline
\end{tabular}
\\
\caption{Numerical simulation of the model given by Eq.~\ref{eqn39} setting as initial conditions on the manifold approximated by  $x_3(0)=-\frac{4}{7}x_{2}^2+\frac{32}{119}x_{1}^2x_{2}-0.2598x_{1}^2x_{2}^2$; the other initial conditions were set to $x_1(0)=0.2$, $x_2(0)=0.2$}
    \label{Table3}
\end{table}

\subsection{Kinetic Monte Carlo Simulation of CO oxidation on a Catalyst}
\label{subsec:MonteCarlo}

  The proposed approach is illustrated through a kMC microscopic
model \cite{Makeev2002} describing the dynamics of CO oxidation on a catalyst. The species react, are adsorbed or desorbed on a finite
lattice with periodic boundary conditions. At each time instant, the sites of the lattice are considered to be either vacant or occupied by the reaction species. The system dynamics are described by the following chemical master equation:

\begin{equation}
\frac{dP(x,t)}{dt}=\sum\limits_{y\neq x} Q(x,y)P(y,t)-\sum\limits_y Q(x,y)P(x,t),
\label{eqn79}
\end{equation}

where $P(x,t)$ is the probability that the system will be in state $x$ at time $t$ and $Q(y,y)$ is the probability for the transition from state $y$ to $x$ per unit time.
The summation runs over all possible transitions (reactions). Here, the numerical simulation of the above stochastic equation was realized using the Gillespie kMC algorithm \cite{Gillespie1976,Gillespie1977}. The reaction mechanism can be schematically described by the following elementary steps:\\
\\
(1) $CO_{gas} +*_i \leftrightarrow CO_{ads,i}$\\
(2) $O_{2,gas}+*_i+*_j \leftrightarrow O_{ads,i} +O_{ads,j}$\\
(3) $CO_{ads,i}+O_{ads,j} \rightarrow CO_{2,gas}+*_i+*_j$\\
\\
where $i$, $j$ are sites on the square lattice, $*$ denotes a site with a vacant adsorption site,
while ``ads" denotes adsorbed particles. By Adding an inert site-blocking adsorbate with a reversible adsorption step the mean field approximation can be derived by the master equation (Eq.~\ref{eqn79}) and is given by the following system of ordinary differential equations \cite{Makeev2002}:

\begin{equation}
\begin{split}
\frac{d \theta_A}{dt}= \alpha(1-\theta_A-\theta_B-\theta_C)-\gamma \theta_A-4k_r \theta_A \theta_B \\
\frac{d \theta_B}{dt}= 2 \beta(1-\theta_A-\theta_B-\theta_C)^2-4k_r \theta_A \theta_B \\
\frac{d\theta_C}{dt}=\mu (1-\theta_A-\theta_B-\theta_C)-\eta \theta_C,
\end{split}
\label{eqn80}
\end{equation}

where $\theta_i$ represent the coverages of species ($i=A,B,C$,
resp. $CO$, $O$ and inert species $C$) on the catalytic surface;
$\mu$ denotes C adsorption and $\eta$ C desorption rate. For $\alpha=1.6$, $\gamma=0.04$, $k_r=1$ , $\eta=0.016$, $\mu=0.36$ and
treating $\beta$ as the bifurcation parameter the mean field model (22) exhibits two Andronov-Hopf points at
$(\theta_{A}^*,\theta_{B}^*,\theta_{C}^*,\beta^*)_1\approx (0.3400,
0.0219, 0.6108, 20.2394)$ and
$(\theta_{A}^*,\theta_{B}^*,\theta_{C}^*,\beta^*)_2 \approx
(0.1895, 0.0575, 0.7207, 21.2779)$. Between the two Andronov-Hopf points, the equilibria are saddles.\\

For the kMC simulations the number of the sites (system size) and the number of-consistent to the mean values of the distribution on the lattice-realizations were chosen to be $N_{size}=800 \times 800$ and $N_{r}=2000$, respectively. The value of the time horizon was selected as $T=0.05$ . The coarse-timestepper of the kMC realizations were used as black box coarse timesteppers. The coarse-grained bifurcation diagram was obtained by applied the Equation-free approach  upon convergence of the Newton-Raphson to a residual of $O(10^{-3})$ for $\epsilon \approx 10^{-2}$. We have chosen this model as for big enough lattice-realizations and runs, the coarse-grained bifurcation diagram and stability practically coincides with the one obtained from the mean filed model; thus one can perform a direct comparison of the numerical approximation of the stable manifold obtained with the kMC simulator and the one derived analytically from the mean field model.\\
Here, we have chosen to find the stable and unstable manifolds at $\beta =20.7$. For this value of the bifurcation parameter, the coarse-grained fixed point is $(\theta_{A}^*,\theta_{B}^*,\theta_{C}^*\approx
(0.2924, 0.0294, 0.6492)$ and the corresponding coarse-grained Jacobian is $~$
$\begin{bmatrix} 0.9246&-0.1209&-0.0683\\
-0.0109&0.8450&-0.0104\\
-0.0156&-0.0151&0.9832
\end{bmatrix}$ (compare this with the one that is obtained from the T-map of the mean field model:  $\begin{bmatrix} 0.9244&-0.1202&-0.0684\\
-0.0109&0.8466&-0.0138\\
-0.0161&-0.0151&0.9830
\end{bmatrix}$). The coarse-grained eigenvalues and corresponding eigenvectors are:\\
$\lambda_1 \approx 0.7515$, $\boldsymbol{v}_1=\left(\begin{array}{c}
-0.5959\\-0.7975\\-0.0942 \end{array} \right)$, $\lambda_{2,3} \approx 1.0006 \pm 0.013i$ , $\boldsymbol{v}_{2,3}= \left(\begin{array}{c} -0.7967\\0.1826 \pm 0.0633i\\ 0.5645 \mp - 0.0967i \end{array} \right)$.\par

\subsubsection{Numerical Parametrization of the Stable Manifold}
A third-order approximation of the stable manifold is given by the following relations:
 
 \begin{equation}
\boldsymbol{h_s(z_s)}\approx \begin{bmatrix}
    a_{1}^{(1)}z_{s}+a_{2}^{(1)}z_{s}^2+a_{3}^{(1)}z_{s}^3\\
    a_{1}^{(2)}z_{s}+a_{2}^{(2)}z_{s}^2+a_{3}^{(2)}z_{s}^3
    \end{bmatrix}
    \label{eqn81}
\end{equation}

 For the numerical approximation of the stable manifold, we have chosen $n_p=6$ points around the coarse-grained saddle. In particular, we have set the following initial values for the $z_s: \{-0.005, -0.003, -0.001, 0.001, 0.003, 0.005\}$ (note that if $z_s>0.03$ the transformation $\boldsymbol{x}'=V\boldsymbol{z}$ results to negative values of $x(2)$); we have also set $k_{max}=2$. By applying the proposed numerical method, the stable manifold is approximated by the following relations:

\begin{equation}
\boldsymbol{h}_s(z_s)\approx \begin{bmatrix}
    -0.0155z_{s}-4.5964z_{s}^2+42.9421z_{s}^3\\
    -0.0797z_{s}-29.0291z_{s}^2+270.0737z_{s}^3
\end{bmatrix}
    \label{eqn82}
\end{equation}

For comparison purposes we also computed the corresponding stable manifold for the mean-field model (~\ref{eqn80}). Following the approach described in section 2, one obtains a nonlinear system of six algebraic equations (see Appendix) which was solved for the unknown coefficients with Newton-Raphson; again the convergence tolerance was of the order of $10^{-3}$ while the perturbation for computing the jacobian matrices was of the order of $10^{-2}$. In this case, the expression for the approximation of the manifold reads:

\begin{equation}
\boldsymbol{h}_s(z_s)\approx \begin{bmatrix}
    -4.6775z_{s}^2+43.2058z_{s}^3\\
   -29.0746z_{s}^2+270.8824z_{s}^3
\end{bmatrix}
    \label{eqn83}
\end{equation}

By comparing the expressions (\ref{eqn82}) \& (\ref{eqn83}), we see that the numerical approximation of the coarse-grained manifold of the kMC simulator is in a good agreement with the  one obtained by the mean field model. \\
One can verify that this is a good approximation of the coarse-grained stable manifold around the coarse-saddle by performing temporal simulations. In table 4 are given various instances of the temporal simulation of the mean field model with initial conditions constrained on the approximation of the mean-field manifold given by Eq.(\ref{eqn80}).

\begin{table}[h!]
    \centering
\begin{tabular}{c|c c c c c c c c c} 
$time$&$0$&$0.25$&$0.5$&$0.75$&$1$&$1.25$&$1.5$&$1.75$&$2$\\
\hline
$\theta_{A}(t)$ &$0.3102$ & $0.2970$&$0.2935$ &$0.2925$& $0.2923$&$0.2922$&$0.2922$&$0.2922$&$0.2922$ \\
$\theta_{B}(t)$&$0.0443$&$0.0357$& $0.0312$&$0.0300$&$0.0297$&$0.0296$&$0.0296$&$0.0296$&$0.0296$\\
$\theta_{C}(t)$&$0.6514$&$0.6499$&$0.6493$& $0.6492$&$0.6492$&$0.6492$&$0.6492$&$0.6492$&$0.6492$\\
\hline
\end{tabular}
\\
\caption{Numerical simulation of the mean field model (~\ref{eqn80}) by constraining the initial conditions on the manifold approximated by Eq.\ref{eqn83}}
    \label{Table4}
\end{table}

In table 5 is shown the numerical simulation of the kMC simulator. The initial conditions were created by lifting the concentrations of the reactants to the coarse-grained manifold given by Eq.\ref{eqn82}.\\

\begin{table}[h!]
    \centering
\begin{tabular}{c|c c c c c c c c c} 
$time$&$0$&$0.25$&$0.5$&$0.75$&$1$&$1.25$&$1.5$&$1.75$&$2$\\
\hline
$\theta_{A}(t)$ &$0.3098$ & $0.2976$&$0.2937$ &$0.2928$& $0.2925$&$0.2925$&$0.2924$&$0.2924$&$0.2924$ \\
$\theta_{B}(t)$&$0.0475$&$0.0359$& $0.0313$&$0.0301$&$0.0297$&$0.0297$&$0.0296$&$0.0296$&$0.0296$\\
$\theta_{C}(t)$&$0.6517$&$0.6496$&$0.6491$& $0.6489$&$0.6489$&$0.6489$&$0.6489$&$0.6489$&$0.6489$\\
\hline
\end{tabular}
\\
\caption{Numerical simulation of the kMC simulator by first constraining the coarse-grained initial conditions on the manifold approximated by Eq.\ref{eqn63} and lifting to appropriate reactant concentrations on the lattice.}
    \label{Table5}
\end{table}

\subsubsection{Numerical Parametrization of the Unstable Manifold}

We seek for the following parametrization of the unstable manifold

 \begin{equation}
 \begin{aligned}
h_u(\boldsymbol{z}_u)\approx
    a_{1,0}z_{u1}+a_{2,0}z_{u1}^2+a_{0,1}z_{u2}+a_{0,2}z_{u2}^2+a_{1,1}z_{u1}z_{u2}+a_{1,2}z_{u1}z_{u2}^2+a_{2,1}z_{u1}^2z_{u2}
    \label{eqn84}
\end{aligned}
\end{equation}

 For the numerical approximation of the unstable manifold, we have chosen again $n_p=6$ points around the coarse-grained saddle. In particular, we have set the following initial values for the $z_{u1}$,$z_{u2}$: $\{-0.05, -0.03, -0.01 0.01 0.03, 0.05\}$ and set $k_{max}=2$. By applying the proposed numerical method, the unstable manifold is approximated as:

\begin{equation}
h_u(\boldsymbol{z}_u)\approx
    -0.1543z_{u1}^2-0.0084z_{u2}^2-0.0817z_{u1}z_{u2}+0.0581z_{u1}z_{u2}^2+0.1596z_{u1}^2z_{u2}
    \label{eqn85}
\end{equation}

For comparison purposes, we also computed the corresponding unstable manifold for the mean-field model (~\ref{eqn80}). Following the approach described in Appendix, we obtained analytically seven algebraic equations (see Appendix) which were solved for the unknown coefficients with Newton-Raphson; the convergence tolerance was of the order of $10^{-6}$ while the perturbation for computing the Jacobian matrices was of the order of $10^{-2}$. In this case, the parametrization of the unstable manifold reads:

\begin{equation}
h_u(\boldsymbol{z}_u)\approx
    -0.1521z_{u1}^2-0.0079z_{u2}^2-0.0747z_{u1}z_{u2}+0.0595z_{u1}z_{u2}^2+0.1419z_{u1}^2z_{u2}
    \label{eqn86}
\end{equation}

By comparing the expressions (\ref{eqn85}) \& (\ref{eqn86}), we see that the numerical approximation of the coarse-grained manifold of the kMC simulator is in a fair agreement with the one obtained by the mean field model. \\

\section{Conlcusions}
\label{sec:Conclusions}
We propose a numerical method for the parametrization of the semi-local coarse-grained stable and unstable manifolds of saddle/stationary points of microscopic simulators when macroscopic models in a closed form in the form of ODEs are not explicitly available. The methodology is based on the Equation-free multiscale framework. The numerical methodology estimates the coefficients of a polynomial expansion of the invariant manifolds by a nonlinear least squares algorithm. The proposed numerical Equation-Free algorithm consists of three steps: (a) detection of the
coarse-grained saddle by constructing the coarse-timestepper of the microscopic dynamics, (b) estimation of the coarse-grained Jacobian and evaluation of its eigenvalues and eigenvectors, and (c) estimation  of the coefficients of the polynomial approximation of the invariant manifolds. The later step involves the construction of a map for the coefficients of the polynomial expansion of the manifold. The key assumption of the methodology is that a macroscopic model in the form of ODEs can in principle describe the emerging macroscopic dynamics but it is not available in a closed form. This assumption implies that there is time-scale separation between the higher-order and lower-order moments of the
evolving microscopic distribution. The proposed numerical approach was illustrated through two examples, a toy model treated as a black-box time-stepper and a kinetic Monte Carlo simulator of a simple catalytic reaction. For the kMC simulator a mean field model in the form of ODEs was also given. For both models, we have also derived analytically a parametrization of the invariant manifolds for comparison purposes.  As we show, the proposed numerical method approximates fairly well the parametrization obtained analytically taking the vector fields as known. \\
The proposed numerical method estimates a parametrization of the stable and unstable manifolds in a neighborhood of the coarse-grained saddle. In a future work, we aim at extending the proposed numerical method to perform a piece-wise parametrization of the global manifold. This could be done for example by coupling the proposed algorithm with an arc-length continuation of the polynomial coefficients as we move far from the equilibrium. 
Another point that requires further investigation in a future work is the analysis of the convergence properties of the algorithm. There are several numerical issues that should be studied such as the convergence properties of the scheme with respect to the amplitude of stochasticity, the sensitivity  of the parametrization with respect to the discretization of the domain around the saddle as well as the issue of finding confidence intervals for the coefficients of the polynomial expansion. 

\subsection*{Author contributions}
C.S. conceived the idea, developed the numerical methodology, performed the numerical analysis and wrote the original manuscript. L.R. verified the mumerical and analytical methods, suggested the examples and performed the numerical simulations. Both authors discussed the results and contributed to the final manuscript

\subsection*{Financial disclosure}

None reported.

\subsection*{Conflict of interest}

The authors declare no potential conflict of interests.

\appendix
\section{Extraction of the Stable and Unstable Manifolds for the Mean Field model of ODEs}
\label{sec:appendixa}

Let us assume a continuous model in the following form of ODEs:

\begin{equation}
\frac{d \boldsymbol{x}}{dt}=\boldsymbol{f}(\boldsymbol{x},\boldsymbol{p}), \boldsymbol{f}:R^n \times R^m \rightarrow R^n 
\label{eqn1ap}
\end{equation}

where $\boldsymbol{f}$  is considered to be sufficiently smooth.

To determine the semi-local stable and unstable manifold of a saddle fixed point $(\boldsymbol{x}^*,\boldsymbol{p}^*)$, the following linear transformation is introduced:

\begin{equation}
\hat{\boldsymbol{x}} \equiv (\boldsymbol{x}-\boldsymbol{x}^*)=\boldsymbol{V}\boldsymbol{z} 
\label{eqn2ap}
\end{equation}

where $\boldsymbol{V}$ is the matrix with columns the  eigenvectors $\boldsymbol{v}_j$ of the Jacobian $\nabla_{x}\boldsymbol{f}(\boldsymbol{x},\boldsymbol{p})$  computed at $(\boldsymbol{x}^*,\boldsymbol{p}^*)$. As in section (2) expanding the right-hand side of  Eq.(\ref{eqn1ap})  around $(\boldsymbol{x}^*,\boldsymbol{p}^*)$ and introducing Eq.(\ref{eqn2ap}) we get:

\begin{equation}
\frac{d\boldsymbol{z}}{dt}=\boldsymbol{V}^{-1}\nabla_{\boldsymbol{x}}\boldsymbol{f}(\boldsymbol{x},\boldsymbol{p})\boldsymbol{V}\boldsymbol{z}+\boldsymbol{V}^{-1}\boldsymbol{g}(\boldsymbol{Vz},\boldsymbol{p}) 
\label{eqn3ap}
\end{equation}

$\boldsymbol{g}(\boldsymbol{Vz},\boldsymbol{p})$ contains the higher order terms with respect to $\boldsymbol{x}$.

By rearranging appropriately the columns of $\boldsymbol{V}$, the Jacobian  $\boldsymbol{J} \equiv \boldsymbol{V}^{-1}\nabla_{\boldsymbol{x}}\boldsymbol{f}(\boldsymbol{x},\boldsymbol{p})\boldsymbol{V}$ can be written in a block form as $\boldsymbol{J}= \left(\begin{array}{cc} \boldsymbol{\Lambda}_s & \boldsymbol{0}\\
\boldsymbol{0} & \boldsymbol{\Lambda}_u \end{array} \right)$, where $\boldsymbol{\Lambda}_s$ is the  $l\times l$ (diagonal/block diagonal) matrix whose eigenvalues are the $l$ eigenvalues with negative real parts and $\boldsymbol{\Lambda}_u$ is the  $n-l\times n-l$ (diagonal/block diagonal) matrix whose eigenvalues are the $n-l$ eigenvalues with positive real parts.Thus, the system given by Eq.(\ref{eqn3ap}) can be written as:

\begin{equation}
\begin{aligned}
\frac{d\boldsymbol{z}_s}{dt}=\boldsymbol{\Lambda}_s\boldsymbol{z}_s+\boldsymbol{g}_s(\boldsymbol{Vz},\boldsymbol{p}) \\
\frac{d\boldsymbol{z}_u}{dt}=\boldsymbol{\Lambda}_u\boldsymbol{z}_u+\boldsymbol{g}_u(\boldsymbol{Vz},\boldsymbol{p})
\end{aligned}
\label{eqn4ap}
\end{equation}

where,

\begin{equation}
\begin{bmatrix} \boldsymbol{g}_s \\ \boldsymbol{g}_u  \end{bmatrix}=\begin{bmatrix} \boldsymbol{V}_1 & \boldsymbol{V}_2  \end{bmatrix}^{-1}\boldsymbol{g}(\boldsymbol{Vz},\boldsymbol{p})
\label{eqn5ap}
\end{equation}

$\boldsymbol{V}_1$ and $\boldsymbol{V}_2$ are the sub-matrices of dimensions $n \times l$ and $n \times n-l$, whose columns contain the eigenvectors corresponding to the eigenvalues with negative and positive real parts, respectively.

A stable manifold  is given by the following equation:
\begin{equation}
\boldsymbol{z}_u = \boldsymbol{h}_s(\boldsymbol{z}_s)
\label{eqn6ap}
\end{equation}

while an unstable manifold  is given by the following equation:
\begin{equation}
\boldsymbol{z}_s = \boldsymbol{h}_u(\boldsymbol{z}_u)
\label{eqn7ap}
\end{equation}

The dynamics on the stable manifold can can be computed by differentiating Eq.(\ref{eqn6ap}) with respect to time to get:

\begin{equation}
\frac{d\boldsymbol{z}_u}{dt} = \nabla_{\boldsymbol{z}_s}\boldsymbol{h}(\boldsymbol{z}_s)\frac{d\boldsymbol{z}_s}{dt}
\label{eqn8ap}
\end{equation}

Given Eq.(~\ref{eqn4ap}), Eq.(\ref{eqn8ap}) becomes:

\begin{equation}
\boldsymbol{\Lambda}_u\boldsymbol{h}(\boldsymbol{z}_s)+\boldsymbol{g}_u(\boldsymbol{z}_s,\boldsymbol{h}_s(\boldsymbol{z}_s),\boldsymbol{p})=\nabla_{\boldsymbol{z}_s}\boldsymbol{h}_s(\boldsymbol{z}_s)[\boldsymbol{\Lambda}_s\boldsymbol{z}_s+\boldsymbol{g}_s(\boldsymbol{z}_s,\boldsymbol{h}_s(\boldsymbol{z}_s),\boldsymbol{p})]
\label{eqn9ap}
\end{equation}
\\
Accordingly, the dynamics on the unstable manifold can can be computed by differentiating Eq.\ref{eqn7ap} with respect ot time to get:

\begin{equation}
\frac{d\boldsymbol{z}_s}{dt} = \nabla_{\boldsymbol{z}_u}\boldsymbol{h}_u(\boldsymbol{z}_u)\frac{d\boldsymbol{z}_u}{dt}
\label{eqn10ap}
\end{equation}

Given Eq.(~\ref{eqn4ap}), Eq.(\ref{eqn10ap}) becomes:

\begin{equation}
\boldsymbol{\Lambda}_s\boldsymbol{h}_u(\boldsymbol{z}_u)+\boldsymbol{g}_s(\boldsymbol{h}_u(\boldsymbol{z}_u),\boldsymbol{z}_u,\boldsymbol{p})=\nabla_{\boldsymbol{z}_u}\boldsymbol{h}_u(\boldsymbol{z}_u)[\boldsymbol{\Lambda}_u\boldsymbol{z}_u+\boldsymbol{g}_u(\boldsymbol{h}_u(\boldsymbol{z}_u),\boldsymbol{z}_u,\boldsymbol{p})]
\label{eqn11ap}
\end{equation}

As described in section 2, the stable and unstable manifolds can be approximated by polynomials, and the coefficients of the terms of the same order in both sides of Eqs.(\ref{eqn4ap}), (\ref{eqn8ap}) are equated. This leads to a system of (nonlinear) algebraic equations (homological equations) to be solved for the unknown polynomial coefficients.\\

As described in section 4.2 we aim at computing the stable and unstable manifolds of the mean field model of CO oxidation given by Eq.(\ref{eqn80}) at $\beta =20.7$. The fixed point is $(\theta_{A}^*,\theta_{B}^*,\theta_{C}^*\approx
(0.2924, 0.0294, 0.6492)$ and the corresponding Jacobian of the right-hand-side of Eq.(\ref{eqn80})  is 

\begin{equation}
\boldsymbol{J}(\theta_{A}^*,\theta_{B}^*,\theta_{C}^*)=
\begin{bmatrix} -1.7578&-2.7698&-1.60\\
-2.5069&-3.5589&-2.3891\\
-0.360&-0.360&-0.376
\end{bmatrix})
\label{eqn12ap}
\end{equation}

The eigenvalues and corresponding eigenvectors of $\boldsymbol{J}(\theta_{A}^*,\theta_{B}^*,\theta_{C}^*)$ are:\\
$\lambda_1 \approx -5.7148$, $\boldsymbol{v}_1=\left(\begin{array}{c}
-0.5961\\-0.7973\\-0.0939 \end{array} \right)$, $\lambda_{2,3} \approx 0.0110 \pm 0.0300i$, $\boldsymbol{v}_{2,3}= \left(\begin{array}{c} -0.7964\\0.1851\pm 0.0729i\\ 0.5600 \mp0.1112i \end{array} \right)$.

\subsection{Parametrization of the Stable Manifold of the Mean Field Model of CO Oxidation on Catalytic Surfaces}

For the mean field model of CO oxidation given by (\ref{eqn80}), we used a third-order approximation of the stable manifold  around $(\theta_{A}^*,\theta_{B}^*,\theta_{C}^*)$ given by:

\begin{equation}
\boldsymbol{h_s(z_s)}\approx \begin{bmatrix}
    a_{1}^{(1)}z_{s}+a_{2}^{(1)}z_{s}^2+a_{3}^{(1)}z_{s}^3\\
    a_{1}^{(2)}z_{s}+a_{2}^{(2)}z_{s}^2+a_{3}^{(2)}z_{s}^3
\end{bmatrix}
    \label{eqn13ap}
\end{equation}

 Introducing Eq.(\ref{eqn13ap}) into Eq.(\ref{eqn11ap})  and equating the terms up to third order of both sides, we get the following set of six nonlinear equations:

\begin{equation}
0.011005a_{1}^{(1)} + 0.030017 a_{1}^{(2)}=-5.7148 a_{1}^{(1)}\\
\label{eqn14ap}
\end{equation}

\begin{equation}
\begin{split}
    (2a_{2}^{(1)}(- 5.7148)- a_{1}^{(1)}(0.8705{a_{1}^{(1)}}^2 + 0.43711a_{1}^{(1)}a_{1}^{(2)} + 2.1719a_{1}^{(1)} + 
    0.047872{{a_{1}^{(2)}}}^2\\
    + 3.9445a_{1}^{(2)}+ 69.473)=(-0.088879{a_{1}^{(1)}}^2 + 0.0356a_{1}^{(1)}a_{1}^{(2)} + 4.2609a_{1}^{(1)} + 0.035831{{a_{1}^{(2)}}}^2 \\+ 2.7342a_{1}^{(2)} +
    0.011005a_{2}^{(1)} +0.030017a_{2}^{(2)} + 54.386)
        \label{eqn15ap}
    \end{split}
\end{equation}
    
 \begin{equation}
\begin{split}
    (4.2609 a_{2}^{(1)} + 0.011005a_{1}^{(2)} + 2.7342a_{2}^{(2)} + 0.030017a_{3}^{(2)} - 0.17776 a_{1}^{(1)}a_{2}^{(1)} + 0.0356a_{1}^{(1)}a_{2}^{(2)} +\\
    0.0356a_{2}^{(1)}a_{1}^{(2)} + 0.071662 a_{1}^{(2)}a_{2}^{(2)})=
    (3a_{3}^{(1)}(-5.7148) - 2a_{2}^{(1)}(0.8705{a_{1}^{(1)}}^2 \\+ 0.43711a_{1}^{(1)}a_{1}^{(2)} + 2.1719a_{1}^{(1)} + 0.047872{a_{1}^{(2)}}^2 +
    3.9445a_{1}^{(2)}+ 69.473)-a_{1}^{(1)}(2.1719a_{2}^{(1)}\\ +
    3.9445a_{2}^{(2)}+ 1.741a_{1}^{(1)}a_{2}^{(1)} + 0.43711a_{1}^{(1)}a_{2}^{(2)} + 0.43711a_{2}^{(1)}a_{1}^{(2)} + 0.095744a_{1}^{(2)}a_{2}^{(2)}))
        \label{eq16ap}
     \end{split}
\end{equation}
    
      \begin{equation}
    - 5.7148a_{1}^{(2)}=(0.011005a_{1}^{(2)} - 0.030017a_{1}^{(1)}) 
    \label{eqn17ap}
    \end{equation}
    
\begin{equation}
\begin{split}
    (0.28781{a_{1}^{(1)}}^2 + 0.54843a_{1}^{(1)}a_{1}^{(2)} + 23.286a_{1}^{(1)} + 0.22083{a_{1}^{(2)}}^2 + 17.097a_{1}^{(2)}- 0.030017a_{2}^{(1)}\\ + 0.011005a_{2}^{(2)} + 332.49)=
    (2a_{2}^{(2)}(- 5.7148) - a_{1}^{(2)}(0.8705{a_{1}^{(1)}}^2 +\\ 0.43711a_{1}^{(1)}a_{1}^{(2)} + 2.1719a_{1}^{(1)} + 0.047872{a_{1}^{(2)}}^2 + 3.9445a_{1}^{(2)} + 69.473))
    \label{eqn18ap}
     \end{split}
\end{equation}
    
   \begin{equation}
\begin{split}
    (3a_{3}^{(2)}(- 5.7148) - 2a_{2}^{(2)}(0.8705{a_{1}^{(1)}}^2 + 0.43711a_{1}^{(1)}a_{1}^{(2)} + 2.1719a_{1}^{(1)} + 0.047872{a_{1}^{(2)}}^2 \\+ 3.9445a_{1}^{(2)} + 69.473)-a_{1}^{(2)}(2.1719a_{2}^{(1)}+3.9445a_{2}^{(2)} + 1.741a_{1}^{(1)}a_{2}^{(1)} + 0.43711a_{1}^{(1)}a_{2}^{(2)}\\ + 0.43711a_{2}^{(1)}a_{1}^{(2)} + 0.095744a_{1}^{(2)}a_{2}^{(2)}))=\\
    (23.286a_{2}^{(1)} - 0.030017a_{3}^{(1)} + 17.097a_{2}^{(2)} + 0.011005a_{3}^{(2)} + 0.57562a_{1}^{(1)}a_{2}^{(1)} + 0.54843a_{1}^{(1)}a_{2}^{(2)} + \\
    0.54843a_{2}^{(1)}a_{1}^{(2)} + 0.44166 a_{1}^{(2)}a_{2}^{(2)})
    \label{eqn19ap}
   \end{split}
\end{equation}

The above system of nonlinear algebraic equations is solved using Newton-Raphson.

\subsection{Parametrization of the Unstable Manifold of the Mean Field Model of CO oxidation on catalytic surfaces}

We parametrized the unstable manifold of the mean field model (\ref{eqn80}) around $(\theta_{A}^*,\theta_{B}^*,\theta_{C}^*)$ using the following series expansion:
 
 \begin{equation}
 \begin{aligned}
h_u(\boldsymbol{z}_u)\approx
    a_{1,0}z_{u1}+a_{2,0}z_{u1}^2+a_{0,1}z_{u2}+a_{0,2}z_{u2}^2+a_{1,1}z_{u1}z_{u2}+a_{1,2}z_{u1}z_{u2}^2+a_{2,1}z_{u1}^2z_{u2}
    \label{eqn20ap}
\end{aligned}
\end{equation}

Introducing Eq.(\ref{eqn20ap}) into Eq.(\ref{eqn13ap}) and equating the terms up to second order of both sides we get the following set of seven nonlinear equations:

\begin{equation}
\begin{aligned}
0.011005a_{1,0}- 0.030017a_{0,1}=- 5.7148a_{1,0}
\end{aligned}
\label{eqn21ap}
\end{equation}

\begin{equation}
\begin{aligned}
0.030017a_{1,0} + 0.011005a_{0,1}=- 5.7148a_{0,1}
\end{aligned}
\label{eqn22ap}
\end{equation}

\begin{equation}
\begin{aligned}
(a_{0,1}*(332.49a_{1,0}^2 + 23.286a_{1,0} + 0.28781)+ a_{1,0}(54.386a_{1,0}^2 + 4.2609a_{1,0}-0.088879) \\+ 0.02201a_{2,0} - 0.030017a_{1,1})=(- 69.473a_{1,0}^2 - 2.1719a_{1,0} - 5.7148a_{2,0} - 0.8705)
\end{aligned}
\label{eqn23ap}
\end{equation}

\begin{equation}
\begin{aligned}
(a_{1,0}(54.386a_{0,1}^2 + 2.7342a_{0,1}+ 0.035831) + a_{0,1}(332.49a_{0,1}^2 + 17.097a_{0,1}+ 0.22083) \\+ 0.02201a_{0,2}+ 0.030017a_{1,1})=(- 69.473a_{0,1}^2 - 3.9445a_{0,1} - 5.7148a_{0,2} - 0.047872)
\end{aligned}
\label{eqn24ap}
\end{equation}

\begin{equation}
\begin{aligned}
(a_{0,1}(23.286a_{0,1} + 17.097a_{1,0}+ 664.98a_{0,1}a_{1,0} + 0.54843)+ a_{1,0}(4.2609a_{0,1} + 2.7342a_{1,0}+\\ 108.77a_{0,1}a_{1,0} + 0.0356) + 0.02201a_{1,1} + 0.060034a_{2,0} - 0.060034a_{0,2})=\\(- 2.1719a_{0,1} - 3.9445a_{1,0} - 5.7148a_{1,1} - 138.95a_{0,1}a_{1,0} - 0.43711)
\end{aligned}
\label{eqn25ap}
\end{equation}

\begin{equation}
\begin{aligned}
(2a_{0,2}(23.286a_{0,1} + 17.097a_{1,0}+ 664.98a_{0,1}a_{1,0} + 0.54843) + 2a_{2,0}(54.386a_{0,1}^2+ \\
2.7342a_{0,1}+ 0.035831) + a_{1,1}(4.2609a_{0,1} + 2.7342a_{1,0}+ 108.77a_{0,1}a_{1,0} + 0.0356) +\\
a_{1,1}(332.49a_{0,1}^2 + 17.097a_{0,1}+ 0.22083) + a_{0,1}(23.286a_{0,2} + 17.097a_{1,1} + 664.98a_{0,1}a_{1,1} +\\
664.98a_{0,2}a_{1,0})+ a_{1,0}(4.2609a_{0,2} + 2.7342a_{1,1}+ 108.77a_{0,1}a_{1,1} + 108.77a_{0,2}a_{1,0}) + 0.011005a_{1,2} +\\ 0.02201*a_{1,2} + 0.060034a_{2,1})=\\(- 2.1719a_{0,2} - 3.9445a_{1,1} - 5.7148a_{1,2} - 138.95a_{0,1}a_{1,1} - 138.95a_{0,2}*a_{1,0})
\end{aligned}
\label{eqn26ap}
\end{equation}

\begin{equation}
\begin{aligned}
(2a_{0,2}(332.49a_{1,0}^2 + 23.286a_{1,0} - 1.8879e-11*a_{2,0} + 0.28781) + a_{1,1}(23.286a_{0,1} + 17.097a_{1,0}+\\ 664.98a_{0,1}a_{1,0} + 0.54843) + a_{1,1}(54.386a_{1,0}^2 + 4.2609a_{1,0}- 0.088879) + 2a_{2,0}(4.2609a_{0,1} +\\ 2.7342a_{1,0} + 108.77a_{0,1}a_{1,0} + 0.0356) + a_{0,1}(23.286a_{1,1} + 17.097a_{2,0}+ 664.98a_{0,1}a_{2,0} + \\664.98a_{1,0}a_{1,1}) + a_{1,0}(4.2609*a_{1,1} + 2.7342a_{2,0}+ 108.77a_{0,1}a_{2,0} + 108.77a_{1,0}a_{1,1}) +\\ 0.02201a_{2,1} + 0.011005a_{2,1} - 0.060034a_{1,2})=\\(-2.1719a_{1,1} - 3.9445a_{2,0}-5.7148a_{2,1} -138.95a_{0,1}a_{2,0} - 138.95a_{1,0}a_{1,1})
\end{aligned}
\label{eqn27ap}
\end{equation}

The above system of nonlinear algebraic equations is solved with Newton-Raphson.
\clearpage

\bibliographystyle{model1-num-names}
\bibliography{sample.bib}







\end{document}